\documentclass[11pt,reqno,a4paper]{amsart}
\usepackage{amssymb,latexsym,amsmath,amsthm}

\theoremstyle{plain}
\newtheorem{thm}{Th\'eor\`eme}[section] 
\newtheorem{prop}[thm]{Proposition}
\newtheorem{lem}[thm]{Lemme} 
\newtheorem{coro}[thm]{Corollaire}
\newtheorem{quest}[thm]{Question}
\newtheorem{fait}[thm]{Fait}

\theoremstyle{remark}
\newtheorem{remark}[thm]{Remarque}

\theoremstyle{definition}

\newcommand{\ZZ}{\mathbb{Z}}

\newcommand{\RR}{\mathbb{R}}
\renewcommand{\SS}{\mathbb{S}}
\newcommand{\NN}{\mathbb{N}}

\newcommand{\DD}{\mathbb{D}}

\newcommand{\Diff}{\mathrm{Diff}}

\renewcommand{\epsilon}{\varepsilon}

\begin{document}

\sloppy


\title[Points fixes communs]{Des points fixes communs pour des diff\'eomorphismes de $\mathbb{S}^2$ qui commutent et pr\'eservent  une mesure de probabilit\'e}
\author{F. B\'eguin, P. Le Calvez, S. Firmo et T. Miernowski}

\address{F. B\'eguin, Universit\'e Paris-Sud, France}
\address{P. Le Calvez, Universit\'e Pierre et Marie Curie, France}
\address{S. Firmo, Universidade Federal Fluminense, Br\'esil}
\address{T. Miernowski, Universit\'e de Aix-Marseille II, France}

\date{\today{}} 



\thispagestyle{empty}

\begin{abstract}
Nous montrons des r\'esultats d'existence de points fixes communs pour des hom\'eomor\-phismes du plan $\RR^2$ ou la sph\`ere $\SS^2$, qui commutent deux \`a deux et pr\'eservent une mesure de probabilit\'e. Par exemple, nous montrons que des $C^1$-diff\'eomorphismes $f_1,\dots,f_n$ de $\mathbb{S}^2$ suffisamment proches de l'identit\'e, qui commutent deux \`a deux, et qui pr\'eservent une mesure de probabilit\'e dont le support n'est pas r\'eduit \`a un point, ont au moins deux points fixes communs.
\end{abstract}

\thanks{Ce travail n'aurait sans doute pas \'et\'e possible sans du soutien financier de l'accord de coop\'eration France-Br\'esil, qui a financ\'e les s\'ejours de F. B., P. L. C. et T. M. \`a l'Universidade Federal Fluminense.}

\maketitle


\section{Introduction}

Le but de cet article est d'\'etablir de nouveaux th\'eor\`emes d'existence de points fixes communs pour des hom\'eomorphismes de surfaces qui commutent. En d'autres termes, on cherche des points fixes pour les actions continues de $\mathbb{Z}^n$ sur les surfaces. Commen\c cons par un bref historique des r\'esultats de ce type. 

\subsection{R\'esultats ``\`a la Lefschetz"}

Le th\'eor\`eme de Poincar\'e-Lefschetz implique que tout hom\'eomorphisme d'une surface de caract\'eristique d'Euler non-nulle, isotope \`a l'identit\'e, admet un point fixe. Il est donc l\'egitime de se demander si $n$ hom\'eomorphismes d'une surface de caract\'eristique d'Euler non-nulle, isotopes \`a l'identit\'e, qui commutent deux \`a deux, ont toujours un point fixe commun. Autrement dit, une action de $\mathbb{Z}^n$ sur une surface de caract\'eristique d'Euler non-nulle via des hom\'eomorphismes isotopes \`a l'identit\'e  poss\`ede un point fixe global.

En 1964, E. Lima a montr\'e que toute action continue de $\mathbb{R}^n$ sur une surface ferm\'ee de caract\'eristique d'Euler non-nulle admet un point fixe (\cite{Lima1964}). Ce r\'esultat a \'et\'e g\'en\'eralis\'e par J.-F. Plante aux actions continues de groupes de Lie nilpotents (\cite{Plante1986}). L'existence de point fixe pour des actions de groupes discrets est bien s\^ur plus d\'elicate \`a montrer. Le premier r\'esultat concernant les actions de $\mathbb{Z}^n$, motiv\'e par une question de th\'eorie des feuilletages\footnote{Un feuilletage proche d'une fibration triviale en tores de dimension $n$ au-dessus d'une surface ferm\'ee de caract\'eristique d'Euler non-nulle poss\`ede-t-il toujours une feuille compacte~?}, a \'et\'e montr\'e par C.\;Bonatti il y a une vingtaine d'ann\'ees.  Nous notons $\Diff^1(S)$ l'ensemble des $C^1$-diff\'eomorphismes d'une surface $S$.

\begin{thm}[Bonatti \cite{Bonatti1989,Bonatti1990}]
\label{theo.Bonatti}
Pour toute surface ferm\'ee $S$ de caract\'eristique d'Euler non-nulle, il existe un voisinage $U$ de l'identit\'e dans $\Diff^1(S)$ (muni de la topologie $C^1$) avec la propri\'et\'e suivante: des \'el\'ements $f_1,\dots,f_n$ de $U$ qui commutent deux \`a deux ont un point fixe commun. 
\end{thm}

Le fait que, dans l'\'enonc\'e ci-dessus, le voisinage $U$ ne d\'epend pas du nombre $n$ de diff\'eomorphismes consid\'er\'es a \'et\'e montr\'e par S.\;Firmo dans~\cite{Firmo2005}. Le r\'esultat de Bonatti a \'et\'e g\'en\'eralis\'e dans plusieurs directions. En particulier, dans le cas o\`u la surface $S$ est la sph\`ere $\mathbb{S}^2$ et o\`u l'entier $n$ est \'egal \`a $2$, M. Handel a consid\'erablement affaibli l'hypoth\`ese de proxi\-mit\'e \`a l'identit\'e n\'ecessaire pour assurer l'existence d'un point fixe commun. Si $f_1,f_2$ sont deux hom\'eomorphismes pr\'eservant l'orientation de la sph\`ere $\mathbb{S}^2$ qui commutent, Handel consid\`ere des isotopies $(f_1^t)_{t\in [0,1]}$ et $(f_2^t)_{t\in [0,1]}$ joignant l'identit\'e respectivement \`a $f_1$ et $f_2$, et note $W(f_1,f_2)$ la classe du lacet $t\mapsto [f_1^t,f_2^t]$ dans $\pi_1(\mathrm{Homeo}(\mathbb{S}^2),\mathrm{Id})\simeq \mathbb{Z}/2\mathbb{Z}$. L'invariant $W(f_1,f_2)$ est nul d\`es lors que $f_1$ et $f_2$ sont suffisamment proches de l'identit\'e (en topologie $C^0$). Handel montre alors~:

\begin{thm}[Handel, \cite{Handel1992}]
\label{theo.Handel}
Deux diff\'eomorphismes $f_1,f_2$ de la sph\`ere $\mathbb{S}^2$, isotopes \`a l'identit\'e, qui commutent, et qui satisfont $W(f_1,f_2)=0$, ont un point fixe commun.
\end{thm}

On remarquera que le th\'eor\`eme d'Handel est faux sans l'hypoth\`ese de nullit\'e de l'invariant $W(f_1,f_2)$~:  deux rotations d'angle $\pi$ et d'axes orthogonaux forment une paire de diff\'eomorphismes de $\mathbb{S}^2$ qui commutent mais n'ont pas de point fixe commun. On notera \'egalement que les arguments de Handel s'appliquent \`a des hom\'eomorphismes, si l'on suppose \emph{a priori} que ceux n'ont qu'un nombre fini de points fixes (ce qui, dans le contexte, est une hypoth\`ese tr\`es forte). 

Il a fallu attendre 2007 pour que le r\'esultat de Handel soit g\'en\'eralis\'e \`a un nombre $n$ arbitraire de diff\'eomorphismes~:

\begin{thm}[Franks, Handel, Parwani, \cite{FranksHandelParwani2007}]
\label{theo.FranksHandelParwani-sphere}
Soient $f_1,\dots,f_n$ des $C^1$-diff\'eomorphismes  de la sph\`ere~$\mathbb{S}^2$, isotopes \`a l'identit\'e, qui commutent deux \`a deux, et $G$ le sous-groupe de $\mathrm{Diff}^1(\mathbb{S}^2)$ engendr\'e par ces diff\'eomorphismes. Alors il existe un sous-groupe~$G'$ d'indice~$2$ dans $G$, tel que  les \'el\'ements de $G'$ 
ont un point fixe commun. Si l'invariant $W(f_i,f_j)$ est nul pour $1\leq i,j\leq n$, alors les diff\'eomorphismes $f_1,\dots,f_n$ ont un point fixe commun. 
\end{thm}

Dans le cas d'une surface ferm\'ee de caract\'eristique d'Euler strictement n\'egative, la situation est plus simple~:

\begin{thm}[Franks, Handel, Parwani, \cite{FranksHandelParwani2007-2}]
\label{theo.FranksHandelParwani-Euler-negative}
Si $S$ est une surface ferm\'ee de caract\'eristique d'Euler strictement n\'egative, alors des $C^1$-diff\'eomorphismes $f_1,\dots,f_n$ de $S$, isotopes \`a l'identit\'e, qui commutent deux \`a deux, ont un point fixe commun. 
\end{thm}

\subsection{Un r\'esultat ``\`a la Brouwer"}

Le th\'eor\`eme des translations planes de Brouwer affirme que ``la dynamique d'un hom\'eomorphisme de $\mathbb{R}^2$, qui pr\'eserve l'orientation et n'a pas de point fixe, ressemble localement \`a celle d'une translation" (voir \cite{Brouwer1912} ou~\cite{Guillou1994}). Par exemple, si $f$ est un hom\'eomorphisme de $\mathbb{R}^2$, qui pr\'eserve l'orientation et n'a pas de point fixe, alors l'orbite de tout point sous l'action de $f$ tend vers l'infini. Un des outils principaux des preuves des th\'eor\`emes~\ref{theo.FranksHandelParwani-sphere} et~\ref{theo.FranksHandelParwani-Euler-negative}  est la g\'en\'eralisation partielle du th\'eor\`eme de Brouwer \`a l'action de $n$ diff\'eomorphismes qui commutent deux \`a deux~: 

\begin{thm}[Franks, Handel, Parwani, \cite{FranksHandelParwani2007}]
\label{theo.FranksHandelParwani-plan}
Soient $f_1,\dots,f_n$ des $C^1$-diff\'eo\-mor\-phis\-mes pr\'eservant l'orientation du plan $\RR^2$ qui commutent deux \`a deux. Si ces diff\'eomorphismes laissent invariant un m\^eme sous-ensemble compact non-vide du plan, alors ils ont un point fixe commun.
\end{thm}

Le th\'eor\`eme des translations planes de Brouwer implique que tout hom\'eomorphisme du plan $\mathbb{R}^2$ qui pr\'eserve l'orientation et laisse invariante une mesure de probabilit\'e admet un point fixe. Ce fait et le th\'eor\`eme~\ref{theo.FranksHandelParwani-plan} soul\`event naturellement la question suivante~:

\begin{quest}
\label{q.1}
Des hom\'eomorphismes $f_1,\dots,f_n$ du plan $\mathbb{R}^2$ qui pr\'eservent l'orientation, commutent deux \`a deux, et laissent invariante une mesure de probabilit\'e, ont-ils un point fixe commun~?
\end{quest}

Le but de notre article est de r\'epondre positivement \`a la question~\ref{q.1}, lorsque le comportement \`a l'infini des hom\'eomorphismes consid\'er\'es est ``suffisamment sympathique". On en d\'eduira des r\'esultats d'existence de deux points fixes communs pour des $C^1$-diff\'eomorphismes de la sph\`ere qui commutent deux \`a deux, et laissent invariante une mesure de probabilit\'e dont le support n'est pas r\'eduit \`a un point. 

\subsection{\'Enonc\'e de nos r\'esultats}  
\label{ss.enonce-resultats}
Nous notons $\mathrm{Homeo}(\mathbb{R}^2)$ le groupe des hom\'eomorphismes de~$\mathbb{R}^2$, muni de la topologie  de la convergence uniforme sur les compacts, et $\mathrm{Homeo}^+(\mathbb{R}^2)$ le sous-groupe  constitu\'e des hom\'eomorphismes pr\'eservant l'orientation. D'apr\`es le th\'eor\`eme de Kneser, $\mathrm{Homeo}^+(\RR^2)$ se retracte sur $\mathrm{SO}(2)$ (voir~\cite{Kneser1926} ou~\cite[th\'eor\`eme 2.9]{LeRoux2001}). Ainsi, tout \'el\'ement $f$ de $\mathrm{Homeo}^+(\RR^2)$ est isotope \`a l'identit\'e dans $\mathrm{Homeo}^+(\RR^2)$. De plus, si $I=(f_t)_{t\in[0,1]}$ est une isotopie de l'identit\'e \`a~$f$, alors toute autre isotopie est homotope \`a $R^kI$ pour un certain $k\in\mathbb{Z}$, o\`u $R$ d\'esigne le lacet dans $\mathrm{SO}(2)$ form\'e des rotations vectorielles d'angle $2\pi t$ pour $t$ allant de $0$ \`a $1$. Pour tout hom\'eomorphisme $f$, nous notons $\mathrm{Fix}(f)$ l'ensemble de ses points fixes.  Nous noterons $d\theta=\frac{1}{2\pi}\, \frac {-ydx+xdy}{x^2+y^2}$ la 1-forme d'angle polaire usuelle sur $\mathbb{R}^2\setminus \{0\}$. 

Soit $f$ un hom\'eomorphisme du plan $\mathbb{R}^2$ pr\'eservant l'orientation, et $I=(f_t)_{t\in[0,1]}$ une isotopie joignant l'identit\'e \`a $f$. Pour tout point $z$ dans $\RR^2$, notons $\gamma_{I,z}:[0,1]\to\RR^2$ le chemin donn\'e par $\gamma_{I,z}(t)=f_t(z)$. Il existe un voisinage $W$ de l'infini dans $\mathbb{R}^2$ tel que, pour $z\in W$, le chemin 
$\gamma_{I,z}$ \'evite l'origine. On peut alors d\'efinir une application
$$\begin{array}{crcl}
\mathrm{Tourne}_I~: &W&\longrightarrow&\RR\\ 
& z&\longmapsto &\int_{\gamma_{I,z}} d\theta 
\end{array}.$$
Si $z$ et $z'$ sont deux points distincts dans $\RR^2$, notons $\gamma_{I,z,z'}:[0,1]\to\RR^2$ le chemin donn\'e par $\gamma_{I,z,z'}(t)=f_t(z)-f_t(z')$. Si $\Delta$ d\'esigne la diagonale de $\RR^2\times\RR^2$, on peut d\'efinir une application
$$ \begin{array}{crcl}
 \mathrm{Enlace}_I~: & (\RR^2\times\RR^2)\setminus\Delta & \longrightarrow & \RR\\ 
& (z,z') & \longmapsto & \int_{\gamma_{I,z,z'}} d\theta
 \end{array}.$$
 La fonction $\mathrm{Tourne}_I$ prend des valeurs enti\`eres sur les points fixes~: si $z\in\mathrm{Fix}(f)$, la quantit\'e $\mathrm{Tourne}_I(z)$ est le nombre alg\'ebrique de tours que fait le point $f_t(z)$ autour de $0$ quand $t$ varie de $0$ \`a $1$. De m\^eme,  la fonction $\mathrm{Enlace}_I$ prend des valeurs enti\`eres sur les couples de points fixes distincts~: si $z,z'\in\mathrm{Fix}(f)$, la quantit\'e $\mathrm{Enlace}_I(z,z')$ est le nombre alg\'ebrique de tours que fait le segment joignant $f_t(z')$ \`a $f_t(z)$ quand $t$ varie de $0$ \`a $1$. Observons \'egalement que la fonction $\mathrm{Enlace}_I$ et le germe en l'infini de la fonction $\mathrm{Tourne}_I$ ne d\'ependent que de la classe d'homotopie de l'isotopie $I$. De plus, si $I'$ est une autre isotopie joignant l'identit\'e \`a $f$, homotope \`a $R^kI$, alors on a 
\begin{eqnarray}
\label{e.change-isotopie-1} \mathrm{Enlace}_{I'} & = & \mathrm{Enlace}_{I}+k \quad \mbox{sur }(\RR^2\times\RR^2)\setminus\Delta\\ 
\label{e.change-isotopie-2} \mathrm{Tourne}_{I'} & = & \mathrm{Tourne}_{I}+k\quad \mbox{sur un voisinage $W$ de l'infini.}
\end{eqnarray}

\bigskip

Nous introduisons maintenant deux propri\'et\'es qui joueront un r\^ole crucial dans notre article~:

\medskip

\noindent \textbf{(P1)} La fonction $\mathrm{Enlace}_I$ est born\'ee sur $\mathrm{Fix}(f)\times\mathrm{Fix}(f)\setminus\Delta$.

\medskip

\noindent \textbf{(P2)} Il existe un voisinage $W$ de l'infini tel que la fonction $\mathrm{Tourne}_I$ est constante sur $\mathrm{Fix}(f)\cap W$. 

\medskip

Les formules~\eqref{e.change-isotopie-1} et~\eqref{e.change-isotopie-2} ci-dessus montrent que les propri\'et\'e~(P1) et~(P2) ne d\'ependent pas de l'isotopie~$I$ choisie, mais uniquement de l'hom\'eomorphisme $f$. Ces deux propri\'et\'es sont \'evidemment v\'erifi\'ees si l'ensemble des points fixes de $f$ est fini~; il est beaucoup plus int\'eressant de noter qu'elles le sont aussi si $f$ est ``de classe $C^1$ \`a l'infini"~:

\begin{prop}
\label{p.extension-C1}
S'il existe une structure diff\'erentielle sur $\mathbb{S}^2=\mathbb{R}^2\sqcup\{\infty\}$ tel que $f$ s'\'etend en un $C^1$-diff\'eomorphisme de $\mathbb{S}^2$, alors les propri\'et\'es (P1) et (P2) sont v\'erifi\'ees.
\end{prop}

Dans l'appendice, nous d\'ecrirons quelques exemples d'hom\'eomorphismes qui satisfont, ou ne satisfont pas, (P1) et/ou (P2). Nous donnerons \'egalement un exemple de $C^1$-diff\'eomorphisme du plan qui satisfait (P1) et (P2), mais ne s'\'etend pas en un $C^1$-diff\'eomorphisme de la sph\`ere. 

\bigskip

Nous sommes maintenant en mesure d'\'enoncer notre r\'esultat principal~:

\begin{thm}
\label{theo.main}
Soit $f\in \mathrm{Homeo}^+(\RR^2)$ un hom\'eomorphisme du plan qui pr\'eserve l'orientation. On suppose que $f$ v\'erifie les propri\'et\'es (P1) et (P2), et laisse invariante une mesure de masse finie $\mu$ dont le support n'est pas contenu dans $\mathrm{Fix}(f)$. Alors~:

\smallskip

\noindent \textbf{1 -} il existe une partie compacte non-vide de $\mathrm{Fix}(f)$ qui est invariante par tout \'el\'ement de $\mathrm{Homeo}(\RR^2)$ qui commute avec $f$ et qui laisse $\mu$ invariante~;

 \smallskip 

 \noindent \textbf{2 -}  si $f$ s'\'etend en un $C^1$-diff\'eomorphisme de $\mathbb{S}^2=\mathbb{R}^2\sqcup\{\infty\}$, il existe une partie compacte non-vide de $\mathrm{Fix}(f)$ qui est invariante par tout  \'el\'ement de $\mathrm{Homeo}(\RR^2)$ qui commute avec $f$.
\end{thm}

\begin{remark}
\label{r.contre-exemple-pas-C1}
Dans l'appendice, nous d\'ecrirons un $C^\infty$-diff\'eomorphisme $f$ du plan $\RR^2$, qui satisfait les propri\'et\'es (P1) et (P2) (mais ne s'\'etend pas en un $C^1$-diff\'eomorphisme de $\mathbb{S}^2$), qui laisse invariante une mesure de masse finie dont le support n'est pas contenu dans $\mathrm{Fix}(f)$, et qui commute avec la translation $(x,y)\mapsto (x+1,y)$. En particulier, si $A$ est une partie non-vide de $\mathrm{Fix}(f)$ qui est invariante par tout \'el\'ement de $\mathrm{Homeo}(\RR^2)$ qui commute avec $f$, alors $A$ n'est pas compacte. Cet exemple montre que l'hypoth\`ese ``$f$ s'\'etend en un $C^1$-diff\'eomorphisme de $\mathbb{S}^2$" est n\'ecessaire dans l'assertion~2 du th\'eor\`eme~\ref{theo.main}. 
\end{remark}

La combinaison de ce th\'eor\`eme avec ceux de Franks, Handel et Parwani pr\'ec\'edemment cit\'es fournit des nouveaux r\'esultats d'existence de points fixes communs. Voici tout d'abord un r\'esultat concernant les $n$-uplets de $C^1$-diff\'eomorphismes du plan~:

\begin{coro}
\label{coro.1}
Soient $f_1,\dots,f_n$ des $C^1$-diff\'eomorphismes du plan $\mathbb{R}^2$, qui pr\'eservent l'orientation, et commutent deux \`a deux. Si ces diff\'eomorphismes satisfont les propri\'et\'es (P1) et (P2), et s'ils laissent invariante une mesure de probabilit\'e, alors ils ont un point fixe commun.
\end{coro}

En effet~: soit le support de $\mu$ est contenu dans l'ensemble des points fixes communs de $f_1,\dots,f_n$, ce qui implique en particulier que les diff\'eomorphismes $f_1,\dots,f_n$ ont au moins un point fixe commun~; soit, il existe $i\in \{1,\dots,n\}$  tel que le support de $\mu$ n'est pas contenu dans l'ensemble des points fixes du diff\'eomorphisme $f_i$, et on obtient l'existence d'un point fixe commun en mettant bout \`a bout le premier point du th\'eor\`eme~\ref{theo.main} et th\'eor\`eme~\ref{theo.FranksHandelParwani-plan}. 

Voici maintenant un r\'esultat concernant les $n$-uplets de $C^1$-diff\'eomorphismes de $\mathbb{S}^2$, qui d\'ecoule imm\'ediatement du second point du th\'eor\`eme~\ref{theo.main} et du th\'eor\`eme~\ref{theo.FranksHandelParwani-plan}~:

\begin{coro}
\label{coro.2}
Soient $f_1,\dots,f_n$ des $C^1$-diff\'eomorphismes de $\mathbb{S}^2$, qui pr\'eservent l'orientation, et commutent deux \`a deux. Si ces diff\'eomorphismes ont un seul point fixe commun, alors, pour tout \'el\'ement $f$ du sous-groupe de $\mathrm{Diff}^1(\mathbb{S}^2)$ qu'ils engendrent, les seules mesures de probabilit\'e $f$-invariantes sont celles dont le support est contenu dans l'ensemble des points fixes de $f$. 
\end{coro} 

Enfin, en mettant bout \`a bout le th\'eor\`eme~\ref{theo.FranksHandelParwani-sphere}, la proposition~\ref{p.extension-C1} et le corollaire~\ref{coro.1}, on obtient imm\'ediatement le r\'esultat suivant~:  

\begin{coro}
\label{coro.3}
Soient $f_1,\dots,f_n$ des $C^1$-diff\'eomorphismes de la sph\`ere $\mathbb{S}^2$ qui pr\'eservent l'orientation et  commutent deux \`a deux, et $G$ le sous-groupe  de $\mathrm{Diff}^1(\mathbb{S}^2)$ engendr\'e par ces diff\'eomorphismes. On suppose que les diff\'eomorphismes $f_1,\dots,f_n$ laissent invariante une mesure de probabilit\'e dont le support n'est pas r\'eduit \`a un point. Alors il existe un sous-groupe $G'$ d'indice~$2$ dans $G$ tel que  les \'el\'ements de $G'$ ont deux points fixes communs. Si l'invariant $W(f_i,f_j)$ est nul pour $1\leq i,j\leq n$, alors les diff\'eomorphismes $f_1,\dots,f_n$ ont deux points fixes communs. 
\end{coro}

On rappelle que l'invariant $W(f_i,f_j)$ est nul d\`es lors que  les diff\'eomorphismes $f_i$ et $f_j$ sont suffisamment proches de l'identit\'e en topologie $C^0$. 

\subsection{Questions et probl\`emes}
Les r\'esultats \'enonc\'es ci-dessus ne semblent pas optimaux de plusieurs points de vue. 

\smallskip

\noindent \textbf{1 - } Notre preuve du th\'eor\`eme~\ref{theo.main} repose fondamentalement sur le fait que l'hom\'eomorphisme $f$ a un comportement contr\^ol\'e \`a l'infini (\emph{i.e.} qu'il satisfait (P1) et (P2), ou qu'il s'\'etent en un $C^1$-diff\'eomorphisme de $\mathbb{S}^2$). Rien n'indique cependant que les \'enonc\'es du th\'eor\`eme~\ref{theo.main} et des corollaires~\ref{coro.1},~\ref{coro.2} et~\ref{coro.3} deviennent faux si on ne fait aucune hypoth\`ese de comportement \`a l'infini. Par exemple, rien n'interdit \emph{a priori} que le corollaire~\ref{coro.3} reste vrai si on travaille avec des hom\'eomorphismes au lieu de $C^1$-diff\'eomorphismes.
 
 \smallskip
 
 \noindent \textbf{2 - } Le th\'eor\`eme des translations planes de Brouwer implique que tout hom\'eomorphisme du plan pr\'eservant l'orientation qui admet un point r\'ecurrent (ou m\^eme un point non-errant) fixe un point. Il n'est donc pas exclu que, dans les \'enonc\'es ci-dessus, on puisse remplacer l'existence d'une mesure de probabilit\'e invariante par la simple existence d'un point r\'ecurrent ou d'un point non-errant. Par exemple, nous ne savons pas r\'epondre \`a la question suivante~: \emph{Soient $f_1,\dots,f_n$ des $C^1$-diff\'eomorphismes du plan $\RR^2$ qui pr\'eservent l'orientation et commutent deux \`a deux. On suppose qu'un \'el\'ement $f$ du groupe engendr\'e par $f_1,\dots f_n$ poss\`ede un point r\'ecurrent (ou non-errant) qui n'est pas fixe, et s'\'etend en un $C^1$-diff\'eomorphisme de $\SS^2$. Les diff\'eomorphismes  $f_1,\dots,f_n$ ont-ils n\'ecessairement un point fixe commun~?}

 \smallskip 
 
\noindent \textbf{3 - } S. Druck, F. Fang et S. Firmo ont montr\'e que le th\'eor\`eme~\ref{theo.Bonatti} de Bonatti se g\'en\'eralise aux cas d'un groupe nilpotent (\cite{DruckFangFirmo2002}). Il est donc naturel de se demander si les corollaires~\ref{coro.1},~\ref{coro.2} et~\ref{coro.3} se g\'en\'eralisent de m\^eme \`a ce cas. Dans un travail r\'ecent (\cite{Mann2011}), K. Mann a montr\'e l'existence de points fixes globaux pour des groupes (non-commutatifs) d'hom\'eomorphismes du plan pr\'eservant l'aire, sous une hypoth\`ese assez forte concernant le diam\`etre des orbites. On peut se demander s'il est possible de combiner les techniques de Mann et les n\^otres afin d'obtenir d'autres r\'esultats d'existences de points fixes. 

\smallskip

\noindent \textbf{ 4 -} Le corollaire~\ref{coro.3} implique en particulier le renforcement suivant du th\'eor\`eme de Bonatti~: \emph{si $f_1,\dots,f_n$ sont des $C^1$-diff\'eomorphismes de la sph\`ere $\mathbb{S}^2$ suffisamment $C^1$-proches de l'identit\'e, qui commutent deux \`a deux, et qui laissent invariante une mesure de probabilit\'e dont le support n'est pas r\'eduit \`a un point, alors $f_1,\dots,f_n$ ont deux points fixes communs.} N\'eanmoins, notre preuve du corollaire~\ref{coro.3} utilise le th\'eor\`eme~\ref{theo.FranksHandelParwani-plan} de Franks, Handel et Parwani qui lui-m\^eme repose sur des techniques relativement sophistiqu\'ees de ``th\'eorie de Thurston non-compacte". Au contraire, la preuve de Bonatti du th\'eor\`eme~\ref{theo.Bonatti} n'utilise que des arguments tr\`es \'el\'ementaires (essentiellement la d\'efinition de la topologie $C^1$, et le fait que l'existence d'une courbe d'indice~1 force l'existence d'un point fixe). Il est donc naturel de se demander si on peut montrer l'\'enonc\'e ci-dessus en n'utilisant que des techniques \'el\'ementaires ``\`a la Bonatti". Nous savons le faire dans le cas particulier $n=2$, mais pas lorsque $n\geq 3$.

\subsection{Plan de l'article}

Dans la section~\ref{s.reduction}, nous ramenons le th\'eor\`eme~\ref{theo.main} \`a un \'enonc\'e concernant les nombres de rotation des orbites  autour des points fixes de $f$. La preuve de cet \'enonc\'e est l'objet de la section~\ref{s.preuve}. Cette preuve utilise diff\'erentes caract\'erisations et propri\'et\'es du nombre de rotation d'une orbite autour d'un point fixe qui seront explicit\'ees dans la section~\ref{s.definitions-alternatives}, ainsi que deux r\'esultats importants de dynamique topologique sur les surfaces, qui seront rappel\'es dans la section~\ref{s.rappels}. La proposition~\ref{p.extension-C1} sera d\'emontr\'ee dans la section~\ref{s.P1-P2-diffeos}. Enfin, dans l'appendice, nous d\'ecrirons quelques exemples d'hom\'eomorphismes du plan qui satisfont, ou pas, les propri\'et\'es~(P1) et/ou~(P2).

\section{Un \'enonc\'e concernant les nombres de rotation autour des points fixes}
\label{s.reduction}

Le but de cette section est de ramener notre th\'eor\`eme principal~\ref{theo.main} \`a un \'enonc\'e technique concernant les nombres de rotation des orbites autour des points fixes. 

\bigskip

Soit $f$ un hom\'eomorphisme du plan pr\'eservant l'orientation. Nous notons $\mathrm{Rec}^+(f)$ l'ensemble des points positivement r\'ecurrent sous l'action de $f$, c'est-\`a-dire l'ensemble des points $z\in\RR^2$ pour lesquels il existe une suite d'entiers strictement croissante $(n_k)_{k\geq 0}$ telle que la suite de points $f^{n_k}(x)$ tend vers~$x$. Consid\'erons maintenant $I=(f_t)_{t\in [0,1]}$ une isotopie de l'identit\'e \`a $f$, et $z'\in\mathrm{Fix}(f)$. On dira que l'orbite d'un point $z\in \mathrm{Rec}^+(f)$ a un {\it nombre de rotation} $\rho_{I,z'}(z)$ autour de $z'$ si, pour toute suite d'entiers strictement croissante $(n_k)_{k\geq 0}$ telle que la suite $f^{n_k}(z)$ converge vers $z$, on a 
$$\lim_{k\to+\infty} \quad \frac{1}{n_k} \sum_{r=0}^{n_k-1} \mathrm{Enlace}_{I}(f^r(z),z') =\rho_{I,z'}(z).$$ 
Le formule~\eqref{e.change-isotopie-1} dans l'introduction montre que l'existence du nombre de rotation $\rho_{I,z'}(z)$ ne d\'epend pas de l'isotopie $I$ choisie, et sa valeur ne d\'epend que de la classe d'homotopie de $I$. 

\medskip

Supposons maintenant que $f$ v\'erifie la propri\'et\'e~(P2). La formule~\eqref{e.change-isotopie-2} montre que l'on peut alors trouver une isotopie $I$ joignant l'identit\'e \`a $f$ telle que la fonction $\mathrm{Tourne}_I$ s'annule sur les points fixes de $f$ proches de l'infini. Une telle isotopie $I$ sera dite {\it adapt\'ee}. Si on suppose de plus que $\mathrm{Fix}(f)$ n'est pas compact, alors il n'existe qu'une seule classe d'homotopie d'isotopies adapt\'ees~; le nombre de rotation $\rho_{I,z'}(z)$, lorsqu'il est d\'efini, ne d\'epend donc pas du choix de $I$ parmi les isotopies adapt\'ees~; on notera $\rho_{z'}(z):=\rho_{I,z'}(z)$ o\`u $I$ est une isotopie adapt\'ee.

\medskip

Pour tout hom\'eomorphisme $f$ du plan $\RR^2$, nous noterons ${\mathcal{M}}(f)$ l'ensemble des mesures bor\'eliennes de masses finies $f$-invariantes. La proposition suivante est la cl\'e de la preuve de notre th\'eor\`eme principal~\ref{theo.main}.

\begin{prop} 
\label{theo.main-2}
Soit $f$ un hom\'eomorphisme du plan $\RR^2$, qui pr\'eserve l'orientation, v\'erifie les propri\'et\'es (P1) et (P2), et tel que $\mathrm{Fix}(f)$ n'est pas compact.

\begin{enumerate}

\item \label{p.1} Soit $g$ un hom\'eomorphisme du plan qui commute avec $f$. Si $z'\in  \mathrm{Fix}(f)$, si $z\in \mathrm{Rec}^+(f)$ et si $\rho_{z'}(z)$ est d\'efini, alors $\rho_{g(z')}(g(z))$ est \'egalement d\'efini, et on a $\rho_{g(z')}(g(z))= \rho_{z'}(z)$ ou $\rho_{g(z')}(g(z))= -\rho_{z'}(z)$ selon que $g$ pr\'eserve ou renverse l'orientation.

\item \label{p.2} Pour toute mesure $\mu\in{\mathcal{M}}(f)$ et tout  point $z'\in\mathrm{Fix}(f)$, la fonction $\rho_{z'}$ est d\'efinie $\mu$-presque partout sur $ \mathrm{Rec}^+(f)\setminus\{z'\}.$

\item \label{p.3} Pour toute mesure $\mu\in{\mathcal{M}}(f)$ et tout r\'eel $a>0$, on a
$$\lim_{z'\to\infty, \,z'\in \mathrm{Fix}(f)} \;\;\mu\left(\{z\in \mathrm{Rec}^+(f)\;,\;\vert\rho_{z'}(z)\vert\geq a\}\right)=0.$$

\item \label{p.4} Supposons que $f$ s'\'etend en un $C^1$-diff\'eomorphisme de $\mathbb{S}^2=\mathbb{R}^2\sqcup\{\infty\}$. Alors pour tout $a>0$, il existe un voisinage~$W$ de l'infini dans $\RR^2$ tel que, pour tout $z'\in  \mathrm{Fix}(f)\cap W$ et  tout $z\in \mathrm{Rec}^+(f)$, on a $\vert\,\rho_{z'}(z)\vert<a$ si $\rho_{z'}(z)$ est d\'efini.

\item \label{p.5} Pour toute mesure $\mu\in{\mathcal{M}}(f)$ dont le support n'est pas contenu dans $\mathrm{Fix}(f)$, il existe un point $z'\in  \mathrm{Fix}(f)$ tel que $$\mu\left(\{z\in  \mathrm{Rec}^+(f)\;,\; \rho_{z'}(z)\not=0\}\right)\not=0.$$
\end{enumerate}
\end{prop}

Expliquons d\`es maintenant pourquoi le th\'eor\`eme~\ref{theo.main} d\'ecoule de cette proposition.

\begin{proof}[Preuve du th\'eor\`eme~\ref{theo.main} en admettant la proposition~\ref{theo.main-2}]
Soit $f$ un hom\'eomorphisme du plan pr\'eservant l'orientation. Supposons, comme dans l'\'enonc\'e du th\'eor\`eme~\ref{theo.main} que $f$ v\'erifie les propri\'et\'es (P1) et (P2), et laisse invariante une mesure de masse finie $\mu$ dont le support n'est pas contenu dans $\mathrm{Fix}(f)$. Rappelons que $\mathrm{Fix}(f)$ n'est pas vide. En effet, si $\mbox{Fix}(f)$ \'etait vide, alors, par le th\'eor\`eme des translations planes de Brouwer (\cite{Guillou1994}), toute orbite de $f$ serait errante~; en particulier, $f$ ne pourrait pr\'eserver aucune mesure de masse finie.

Le th\'eor\`eme~\ref{theo.main} est trivial si $\mathrm{Fix}(f)$ est compact~: en effet, $\mathrm{Fix}(f)$ est non-vide et invariant par tout hom\'eomorphisme qui commute avec $f$. Nous supposons donc maintenant que $\mathrm{Fix}(f)$ n'est pas compact, ce qui nous permet d'appliquer la proposition~\ref{theo.main-2}.

Pour $a>0$, consid\'erons le sous-ensemble $X_a$ de $\mathrm{Fix}(f)$ d\'efini comme suit~:
$$X_a=\left\{z'\in \mathrm{Fix}(f)\;,\; \mu\left(\{z\in  \mathrm{Rec}^+(f)\;,\;\vert\rho_{z'}(z)\vert\geq a\}\right)\geq a\right\}.$$
L'assertion~(5) de la proposition~\ref{theo.main-2} nous dit que, pour $a$ assez petit, $X_a$ n'est pas vide. L'assertion~(3) nous dit que cet ensemble est born\'e. Enfin, d'apr\`es l'assertion~(1), l'ensemble $X_a$ est invariant par tout hom\'eomorphisme $g$ qui pr\'eserve la mesure $\mu$ et qui commute avec $f$. Pour montrer la premi\`ere assertion du th\'eor\`eme~\ref{theo.main}, il suffit donc de consid\'erer l'adh\'erence de $X_a$, pour $a$ assez petit.

Supposons maintenant que $f$ peut \^etre \'etendu en un $C^1$-diff\'eomorphisme de $\SS^2$, et consid\'erons, pour $a>0$, le sous-ensemble $Y_a$ de $\mathrm{Fix}(f)$ d\'efini comme suit~:
$$Y_a=\left\{z'\in \mathrm{Fix}(f)\; , \; \mathrm {il\enskip existe}\enskip  z\in  \mathrm{Rec}^+(f) \enskip\mathrm{tel\enskip que}\enskip\rho_{z'}(z) \enskip \mathrm{existe\enskip et} \enskip \vert\rho_{z'}(z)\vert\geq a\right\}.$$
L\`a-encore, on sait gr\^ace \`a l'assertion~(5) de la proposition~\ref{theo.main-2}  que $Y_a$ est non vide d\`es que $a$ est assez petit. L'assertion~(4) nous dit que cet ensemble est born\'e, et l'assertion~(1) qu'il est invariant par tout hom\'eomorphisme  qui commute avec $f$. Pour montrer la seconde assertion du th\'eor\`eme~\ref{theo.main}, il suffit donc de consid\'erer l'adh\'erence de $Y_a$, pour $a$ assez petit.
\end{proof}

La preuve de la proposition~\ref{theo.main-2} occupe presque toute la fin de notre article.

\section{D\'efinitions alternatives et propri\'et\'es du nombre de rotation}
\label{s.definitions-alternatives}

Soit $f$ un hom\'eomorphisme du plan $\RR^2$ pr\'eservant l'orientation, et $I=(f_t)_{t\in [0,1]}$ une isotopie joignant l'identit\'e \`a  $f$.   Dans la section~\ref{s.reduction}, nous avons d\'efini le nombre de rotation $\rho_{I,z'}(z)$ de l'orbite d'un point $z\in\mathrm{Rec}^+(f)$ autour d'un point $z'\in\mathrm{Fix}(f)$, comme une moyenne de Birkhoff de la fonction $z\mapsto\mathrm{Enlace}_I(z,z')$ le long de l'orbite de $z$. Ce nombre de rotation admet cependant plusieurs d\'efinitions alternatives. Le but de cette section est de pr\'esenter certaines de ces d\'efinitions, qui interviendront au cours de notre preuve de la proposition~\ref{theo.main-2}. Ce sera \'egalement l'occasion d'introduire un certain nombre de notations, et de mettre en lumi\`ere quelques propri\'et\'es \'el\'ementaires de la quantit\'e $\rho_{I,z'}(z)$.

\bigskip

Nous consid\'erons d'abord le cas particulier o\`u $z$ est un point fixe de $f$. La d\'efinition du nombre de rotation $\rho_{I,z'}(z)$ se simplifie radicalement dans ce cas particulier~:

\begin{fait}
\label{f.nombre-rotation-1}
Si $z$ et $z'$ sont deux points fixes distincts de $f$, alors $\rho_{I,z'}(z)$ existe et est \'egal \`a~$\mathrm{Enlace}_I(z,z')$.
\end{fait}

Remarquons maintenant qu'\'etant donn\'e un point $z'\in \mathrm{Fix}(f)$, on peut toujours trouver une isotopie $I'=(f'_t)_{t\in [0,1]}$ joignant l'identit\'e \`a $f$, homotope \`a $I$, et qui fixe le point $z'$ (\emph{i.e.} telle que $f'_t$ fixe $z'$ pour tout $t$)~: il suffit par exemple de poser $f'_t=\tau_t\circ f_t$ o\`u $\tau_t$ est l'unique translation de $\mathbb{R}^2$ qui envoie $f_t(z')$ sur $z'$. Comme $I$ et $I'$ sont homotopes, on a alors $\mathrm{Enlace}_I(z,z')=\mathrm{Enlace}_{I'}(z,z')$. Par ailleurs, la quantit\'e $\mathrm{Enlace}_{I'}(z,z')$ est, par d\'efinition, le nombre de tours que fait le lacet $t\mapsto f'_t(z)-f'_t(z')$ autour de $0$. Comme $f'_t(z')=z'$, c'est aussi le nombre de tours que fait la trajectoire $\gamma_{I',z}:t\mapsto f'_t(z)$ autour du point $z'$. 

Expliquons maintenant comment interpr\'eter $\rho_{I,z'}(z)$ via un rel\`evement appropri\'e de $f$ au rev\^etement universel de l'anneau $\mathbb{R}^2\setminus\{z'\}$. Pour $z'\in\mathrm{Fix}(f)$, nous consid\'ererons l'anneau $A_{z'}=\mathbb{R}^2\setminus\{z'\}$. Nous noterons  $\pi_{z'}~:\widetilde A_{z'}\to A_{z'}$ le rev\^etement universel de l'anneau $A_{z'}$, et $T_{z'}~:\widetilde  A_{z'}\to \widetilde A_{z'}$ le g\'en\'erateur du groupe des automorphismes de rev\^etement naturellement d\'efini par le bord orient\'e d'un disque euclidien centr\'e en $z'$. \`A l'isotopie $I$ est naturellement associ\'e un rel\`evement $\widetilde f_{I,z'}$ de $f\vert_{A_{z'}}$ \`a $\widetilde A_{z'}$. Pour construire ce rel\`evement, on choisit une isotopie $I'=(f'_t)_{t\in[0,1]}$ homotope \`a $I$ qui fixe le point $z'$, et on rel\`eve l'isotopie $(f'_t\vert_{A_{z'}})_{t\in[0,1]}$ \`a $\widetilde A_{z'}$ en une isotopie partant de l'identit\'e~; l'extr\'emit\'e de cette isotopie est, par d\'efinition, l'hom\'eomorphisme $\widetilde f_{I,z'}$. On v\'erifie facilement  que $\widetilde f_{I,z'}$ ne d\'epend pas du choix de l'isotopie $I'$, et ne d\'epend d'ailleurs que de la classe d'homotopie de $I$. Le fait ci-dessous d\'ecoule directement de la d\'efinition du rel\`evement $\widetilde f_{I,z'}$ et du fait que $\mathrm{Enlace}_I(z,z')=\mathrm{Enlace}_{I'}(z,z')$ si $I$ et $I'$ sont homotopes.

\begin{fait}
\label{f.nombre-rotation-3}
Si $z$ et $z'$ sont deux points fixes distincts de $f$, il existe un unique entier $p$ tel que $\widetilde f_{I,z'}(\widetilde z)=T_{z'}^p(\widetilde z)$ pour tout relev\'e $\widetilde z$ de $z$ dans $\widetilde A_{z'}$. Cet entier $p$ n'est autre que le nombre de rotation $\rho_{I,z'}(z)=\mathrm{Enlace}_I(z,z')$.
\end{fait}

Consid\'erons maintenant trois points fixes distincts $z,z'_1,z'_2$ de $f$. D'apr\`es la formule~\eqref{e.change-isotopie-1}, la quantit\'e $\rho_{I,z'_1}(z)-\rho_{I,z'_2}(z)=\mathrm{Enlace}_I(z,z'_1)-\mathrm{Enlace}_I(z,z'_2)$ ne d\'epend pas de l'isotopie $I$. 

Rappelons que $\mathbb{S}^2$ est vu comme le compactifi\'e $\mathbb{R}^2\sqcup\{\infty\}$. \'Etant donn\'es deux points distincts $z'_1,z'_2\in\mathrm{Fix}(f)$, nous consid\'ererons l'anneau $A_{z'_1,z'_2}:=\mathbb{S}^2\setminus\{z'_1,z'_2\}$. Nous noterons $\pi_{z'_1,z'_2}~:\widetilde A_{z'_1,z'_2}\to A_{z'_1,z'_2}$ le rev\^etement universel de l'anneau $A_{z'_1,z'_2}$, et $T_{z'_1,z'_2}~:\widetilde  A_{z'_1,z'_2}\to \widetilde A_{z'_1,z'_2}$ l'automorphisme de rev\^etement  naturellement d\'efini par le bord orient\'e d'un petit disque euclidien centr\'e en $z'_1$. Nous noterons $\widetilde f_{z'_1,z'_2}$ le rel\`evement de $\overline f\vert_{ A_{z'_1,z'_2}}$ qui fixe les relev\'es de $\infty$, o\`u $\overline f$ est l'extension de $f$ \`a la sph\`ere $\SS^2=\RR^2\sqcup\{\infty\}$. Si $I'=(f_t')_{t\in [0,1]}$ est une isotopie joignant l'identit\'e \`a $f$ et qui fixe les points $z_1'$ et $z_2'$, alors l'extension de cette isotopie \`a $\SS^2$ nous donne une isotopie dans $\mathrm{Homeo}(\SS^2)$ joignant l'identit\'e \`a $\bar f$, qui fixe $z'_1$, $z'_2$ et $\infty$. Si on restreint cette isotopie \`a $A_{z'_1,z'_2}$, puisqu'on rel\`eve au rev\^etement  universel $\widetilde A_{z'_1,z'_2}$, on obtient une isotopie joignant l'identit\'e \`a $\widetilde f_{z'_1,z'_2}$. Ceci montre le fait suivant~:

\begin{fait}
\label{f.nombre-rotation-5}
Si $z$, $z'_1$ et $z'_2$ sont trois points fixes distincts de $f$, alors il existe un entier $p\in\ZZ$ tel que $\widetilde f_{z'_1,z'_2}(\widetilde z)=T_{z'_1,z'_2}^p(\widetilde z)$ pour tout relev\'e $\widetilde z$ de $z$ \`a $\widetilde A_{z'_1,z'_2}$~; cet entier $p$ n'est autre que la diff\'erence des nombres de rotations $\rho_{I,z'_1}(z)-\rho_{I,z'_2}(z)$.
\end{fait}

Avant de passer au cas o\`u $z$ n'est pas un point fixe, relions le nombre de rotation $\rho_{I,z'}(z)$ \`a la quantit\'e $\mathrm{Tourne}_I(z)$.

\begin{fait}
\label{f.nombre-rotation-6}
\'Etant donn\'e $z'\in \mathrm{Fix}(f)$, il existe un voisinage $W$ de l'infini tel que, pour tout $z\in W\cap\mathrm{Fix}(f)$, le nombre de rotation $\rho_{I,z'}(z)$ est \'egal \`a la quantit\'e $\mathrm{Tourne}_I(z)$.
\end{fait}

En particulier, si $I$ est une isotopie adapt\'ee, alors, pour tout $z'\in\mathrm{Fix}(f)$, il existe un voisinage $W$ de l'infini tel que, si $z\in\mathrm{Fix}(f)\cap W$, alors $\rho_{z'}(z)=0$. 

\begin{proof}[Preuve du fait~\ref{f.nombre-rotation-6}]
Soit $B$ une boule euclidienne de $\mathbb{R}^2$ centr\'ee \`a l'origine tel que le lacet $\gamma_{I,z'}:t\mapsto f_t(z')$ soit contenu dans $B$. Soit $W$ un voisinage de l'infini dans $\mathbb{R}^2$ tel que, si $z\in \mathrm{Fix}(f)\cap W$, alors le lacet $\gamma_{I,z}:t\mapsto f_t(z)$ est disjoint de $B$. Alors, pour tout $z\in\mathrm{Fix}(f)\cap W$, les lacets $\gamma_{I,z,z'}:t\mapsto f_t(z)-f_t(z')$ et $\gamma_{I,z}:t\mapsto f_t(z)$ sont homotopes dans $\mathbb{R}^2\setminus\{0\}$, et on a donc 

\noindent~\hfill$\displaystyle \rho_{I,z'}(z)=\mathrm{Enlace}_I(z,z')=\int_{\gamma_{I,z,z'}}d\theta=\int_{\gamma_{I,z}}d\theta=\mathrm{Tourne}_I(z).$\hfill
\end{proof}

\begin{remark}
\label{r.caracterisation-isotopie-adaptee}
Soit $z'$ un point fixe de $f$. Les faits~\ref{f.nombre-rotation-3} et~\ref{f.nombre-rotation-6} montrent que l'isotopie $I$ est adapt\'ee si et seulement si il existe un voisinage $W$ de l'infini tel que le rel\`evement $\widetilde f_{I,z'}:\widetilde A_{z'}\to\widetilde A_{z'}$ fixe les relev\'es des points de $W\cap\mathrm{Fix}(f)$.
\end{remark}

Passons maintenant au cas o\`u $z$ n'est pas un point fixe. Le fait suivant g\'en\'eralise le fait~\ref{f.nombre-rotation-3} et d\'ecoule directement des d\'efinitions du nombre de rotation $\rho_{I,z'}(z)$ et du rel\`evement $\widetilde f_{I,z'}$.

\begin{fait}
\label{f.nombre-rotation-7}
Soit $z'\in\mathrm{Fix}(f)$, soit $z\in\mathrm{Rec}^+(f)$, et $\widetilde z$ un relev\'e de $z$ dans $\widetilde A_{z'}$. Pour toute suite strictement croissante d'entiers $(n_k)_{k\geq 0}$ telle que $f^{n_k}(z)$ tend vers $z$, il existe une suite d'entiers $(p_k)_{k\geq 0}$ telle que $T_{z'}^{-p_k}\circ \widetilde f_{I,z'}^{n_k}(\widetilde z)$ converge vers $\widetilde z$. Dire que l'orbite du point $z$ a un nombre de rotation $\rho_{I,z'}(z)$ autour de $z'$, c'est dire que le quotient $p_k/n_k$ tend vers $ \rho_{I,z'}(z)$ quand $k\to\infty$, quelle que soit la suite d'entiers $(n_k)$. 
\end{fait}

Passons maintenant \`a la g\'en\'eralisation du fait~\ref{f.nombre-rotation-5}.

\begin{fait}
\label{f.nombre-rotation-8}
Soient $z'_1$ et $z'_2$ deux points fixes distincts de $f$. Soit $z\in\mathrm{Rec}^+(f)$ tel que les nombres de rotation $\rho_{I,z'_1}(z)$ et $\rho_{I,z'_2}(z)$ existent tous les deux, et soit $\widetilde z$ un relev\'e de $z$ dans $\widetilde A_{z'_1,z'_2}$.  Pour toute suite strictement croissante d'entiers $(n_k)_{k\geq 0}$ telle que $f^{n_k}(z)$ tend vers $z$, il existe une suite d'entiers $(p_k)_{k\geq 0}$ telle que $T_{z'_1,z'_2}^{-p_k}\circ \widetilde f_{z'_1,z'_2}^{n_k}(\widetilde z)$ converge vers $\widetilde z$. La suite $(p_k/n_k)_{k\geq 0}$ converge vers $ \rho_{I,z'_1}(z)-\rho_{I,z_2'}(z)$. 
\end{fait}

\section{Deux outils}
\label{s.rappels}

Dans cette section, nous pr\'esentons deux r\'esultats, concernant la dynamique des hom\'eomorphismes de surfaces, qui seront des outils importants dans notre preuve de la proposition~\ref{theo.main-2}.

\bigskip

Commen\c cons par rappeler un r\'esultat, d\^u \`a J. Franks, qui constitue une des \'etapes principales de la preuve du th\'eor\`eme des translations planes de Brouwer. Notons 
$$\begin{array}{rrcl}
\pi~: &\RR^2&\to&\mathbb{A}\\ &(x,y)&\mapsto & (x+\ZZ,y)
\end{array}$$ 
le rev\^etement universel de l'anneau ouvert $\mathbb{A}=\RR/\ZZ\times \RR$, et 
$$\begin{array}{rrcl}
T~: &\RR^2&\to&\RR^2\\ &(x,y)&\mapsto &(x+1,y)
\end{array}$$ 
l'automorphisme fondamental de rev\^etement. Nous appellerons {\it disque ouvert} $U$ d'une surface~$M$ toute partie ouverte de $M$ hom\'eomorphe \`a $\RR^2$.

\begin{thm}[Lemme de Franks, \cite{Franks1988}]
\label{theo.lemme-de-Franks}
Soit $g$ un hom\'eomorphisme de  $\mathbb{A}$ isotope \`a l'identit\'e et $\widetilde g$ un rel\`evement de $g$ \`a $\RR^2$. Soit $U$ un disque ouvert de $\mathbb{A}$ et $\widetilde U$ une composante connexe de $\pi^{-1}(U)$. On suppose que
\begin{enumerate}
\item $\widetilde g(\widetilde U)\cap \widetilde U=\emptyset$,

\item il existe $q>0$ et $p\geq 0$ tel que $\widetilde g^q(\widetilde U)\cap T^p(\widetilde U)\not=\emptyset$,

\item il existe $q'>0$ et $p'\leq 0$ tel que $\widetilde g^{q'}(\widetilde U)\cap T^{p'}(\widetilde U)\not=\emptyset$.
\end{enumerate}
Alors $\widetilde g$ a au moins un point fixe.
\end{thm}

\'Enon\c cons maintenant un r\'esultat plus r\'ecent d\^u \`a O.~Jaulent.

\begin{thm}[Jaulent, \cite{Jaulent2010}]
\label{theo.Brouwer-feuillete-equivariant}
Soit $M$ une surface orient\'ee, $g$ un hom\'eomorphisme de $M$ isotope \`a l'identit\'e, et $I=(g_t)_{t\in[0,1]}$ une isotopie joignant l'identit\'e \`a $g$ dans $\mathrm{Homeo}(M)$. Il existe alors une partie ferm\'ee $X\subset \mathrm{Fix}(g)$, et une isotopie $I'=(g'_t)_{t\in [0,1]}$ joignant l'identit\'e \`a $g\vert_{M\setminus X}$ dans $\mathrm{Homeo}(M\setminus X)$ tels que
\begin{enumerate}
\item Pour tout $z\in X$, le lacet $\gamma_{I,z}:t\mapsto g_t(z)$ est homotope \`a z\'ero dans $M$. 

\item Pour tout $z\in\mathrm{Fix}(g)\setminus X$, le lacet $\gamma_{I',z}:t\mapsto g'_t(z)$ n'est pas homotope \`a z\'ero dans $M\setminus X$.

\item Pour tout $z\in M\setminus X$, les trajectoires $\gamma_{I,z}$ et $\gamma_{I',z}$ sont homotopes (\`a extr\'emit\'es fix\'ees) dans~$M$.

\item Il existe un feuilletage topologique orient\'e $\mathcal{F}$ sur $M\setminus X$ tel que, pour tout $z\in M\setminus X$, la trajectoire $\gamma_{I',z}$ est homotope dans $M\setminus X$ \`a un chemin positivement transverse \`a $\mathcal{F}$.
\end{enumerate}

De plus , l'isotopie $I'=(g'_t)_{t\in [0,1]}$ satisfait la propri\'et\'e suivante~: 
\begin{enumerate}
\item[(5)] Pour toute partie finie $Y\subset X$, il existe une isotopie $I'_Y=(g'_{Y,t})_{t\in [0,1]}$ joignant l'identit\'e \`a $g$ parmi les hom\'eomorphismes de $M$ qui fixent les points de $Y$, et telle que, si $z\in M\setminus X$, les chemins $t\mapsto g'_t(z)$ et $t\mapsto g'_{Y,t}(z)$ sont homotopes dans $M\setminus Y$.
\end{enumerate}
\end{thm}

L'assertion~(4)  signifie que la trajectoire $\gamma_{I',z}:t\mapsto g_t'(z)$ est homotope \`a un chemin $\gamma_z~:[0,1]\to M\setminus X$ v\'erifiant la propri\'et\'e suivante. Pour tout $t_0\in[0,1]$, il existe un voisinage $\iota\subset [0,1]$ de $t_0$, un voisinage $U\subset M\setminus X$ de $\gamma_z(t_0)$ et un hom\'eomorphisme $\Phi~: U\to ]0,1[^2$ transportant l'orientation de $M$ sur l'orientation usuelle de $\RR^2$, envoyant le feuilletage $\mathcal{F}\vert_W$ sur le feuilletage en horizontales dirig\'ees suivant les $x$ croissants et tel que $p_2\circ\Phi\circ \gamma_z$ soit croissant sur $\iota$, o\`u $p_2:~\RR^2\to\RR$ est la deuxi\`eme projection. 

\'Etant donn\'es un hom\'eomorphisme $g$ d'une surface $M$, et une isotopie $I$ joignant  l'identit\'e \`a $g$, un point fixe $z\in\mathrm{Fix}(g)$ est dit \emph{contractile} pour l'isotopie $I$ si le lacet $\gamma_{I,z}:t\mapsto g_t(z)$ est contractile dans $M$. L'assertion~(1) du th\'eor\`eme ci-dessus affirme donc que les points de $X$ sont des points fixes contractiles pour l'isotopie $I$.  L'assertion~(2) dit que $I'$ n'a aucun point fixe contractile. L'assertion~(4) n'est autre que la version feuillet\'ee \'equivariante du th\'eor\`eme des translations planes de Brouwer, due \`a P. Le Calvez (\cite{LeCalvez2005}), appliqu\'ee \`a l'isotopie $I'$ sur $M\setminus X$. Tr\`es grossi\`erement, le th\'eor\`eme de Jaulent affirme donc que, quitte \`a \^oter \`a $M$ un ensemble de points fixes contractiles $X$ pour $I$, et quitte \`a remplacer $I$ par une isotopie $I'$ tel que les chemins $\gamma_{I,z}$ et $\gamma_{I',z}$ sont homotopes dans $M$, il n'existe plus aucun point fixe contractile. On peut donc  appliquer le th\'eor\`eme des translations planes feuillet\'ee \'equivariant de Le Calvez.

On ne sait pas s'il est possible de prolonger l'isotopie $I'=(g_t')_{t\in [0,1]}$ donn\'ee par le th\'eor\`eme~\ref{theo.Brouwer-feuillete-equivariant} en une isotopie joignant l'identit\'e \`a $g$ dans $\mathrm{Homeo}(M)$ qui fixe les points de $X$ (sauf dans certains cas particuliers, par exemple si $X$ est totalement discontinu). \`A d\'efaut, l'assertion~(5) du th\'eor\`eme~\ref{theo.Brouwer-feuillete-equivariant} garantit une propri\'et\'e plus faible, mais suffisante pour la plupart des applications.

\section{Preuve de la proposition~\ref{theo.main-2}}
\label{s.preuve}

Dans toute cette partie, nous consid\'erons un hom\'eomorphisme $f$ du plan $\mathbb{R}^2$ qui pr\'eserve l'orientation, qui v\'erifie les propri\'et\'es (P1) et (P2), et dont l'ensemble des points fixes n'est pas compact. Nous fixons \'egalement une isotopie $I=(f_t)_{t\in[0,1]}$ adapt\'ee \`a $f$. On aura donc $\rho_{z'}(z)=\rho_{I,z'}(z)$ d\`es que ces quantit\'es seront d\'efinies. Nous utiliserons les notations d\'efinies dans la partie~\ref{s.definitions-alternatives}. En particulier, nous rappelons que, pour tout point $z'\in\RR^2$, nous notons $A_{z'}$ l'anneau $\RR^2\setminus\{z'\}$, nous notons $\pi_{z'}:\widetilde A_{z'}\to A_{z'}$ le rev\^etement universel de cet anneau, et $\widetilde f_{I,z'}$ le rel\`evement de $f\vert_{A_{z'}}$ \`a $\widetilde A_{z'}$ associ\'e \`a l'isotopie $I$.

\medskip

\begin{proof}[Preuve de l'assertion~(1).]
Consid\'erons un hom\'eo\-morphisme~$g$ du plan qui commute avec $f$, un point $z'\in\mathrm{Fix}(f)$ et un point $z\in\mathrm{Rec}^+(f)$, distinct de $z'$, tels que le nombre de rotation $\rho_{z'}(z)$ existe. Supposons, pour fixer les id\'ees, que $g$ pr\'eserve l'orientation. Choisissons un rel\`evement  $\widetilde g:\widetilde A_{z'}\to \widetilde A_{g(z')}$ de l'hom\'eo\-morphisme $g\vert _{A_{z'}}:A_{z'}\to A_{g(z')}$. 

Nous affirmons que les deux \'egalit\'es suivantes ont lieu~:
\begin{equation}
\label{e.commutation}
\widetilde g \circ T_{z'} = T_{g(z')}\circ \widetilde g\quad\quad\quad\widetilde g \circ \widetilde f_{I,z'} = \widetilde f_{I,g(z')}\circ \widetilde g.
\end{equation}
L'\'egalit\'e de gauche d\'ecoule directement des d\'efinitions des automorphismes de rev\^etement $T_{z'}$ et $T_{g(z')}$, et du fait que $g$ pr\'eserve l'orientation. Pour montrer l'\'egalit\'e de droite, rappelons que $\widetilde f_{I,z'}$ est le rel\`evement de $f\vert_{A_{z'}}$ \`a $\widetilde A_{z'}$ qui fixe les relev\'es des points de $\mathrm{Fix}(f)$ suffisamment proches de l'infini (voir la remarque~\ref{r.caracterisation-isotopie-adaptee}). Si $z$ est un point de $\mathrm{Fix}(f)$ proche de l'infini, alors $g^{-1}(z)$ est aussi un point de $\mathrm{Fix}(f)$ proche de l'infini. Par suite, si $\widetilde z$ est un relev\'e de $z$ dans $\widetilde A_{g(z')}$, alors $\widetilde f_{I,z'}$ fixe $\widetilde g^{-1}(\widetilde z)$, et on a donc $\widetilde g \circ \widetilde f_{I,z'} \circ \widetilde g^{-1}(\widetilde z)=\widetilde z$. Ceci montre que $\widetilde g \circ \widetilde f_{I,z'} \circ \widetilde g^{-1}$ est le rel\`evement de $f\vert_{A_{g(z')}}$ \`a $\widetilde A_{g(z')}$ qui fixe les relev\'es des points de $\mathrm{Fix}(f)$ suffisamment proches de l'infini. D'apr\`es la remarque~\ref{r.caracterisation-isotopie-adaptee}, ce rel\`evement co\"{\i}ncide donc avec $\widetilde f_{I,g(z')}$. Ceci montre l'\'egalit\'e de droite.

Consid\'erons alors une suite d'entiers strictement croissante $(n_k)_{k\geq 0}$. Puisque $g$ commute avec $f$, si la suite de points $(f^{n_k}(g(z)))_{k\geq 0}$ converge vers $g(z)$, alors la suite de points $(f^{n_k}(z))_{k\geq 0}$ converge vers~$z$. Puisque $\widetilde g$ est un rel\`evement de $g$, si le point $\widetilde g(\widetilde z)\in \widetilde A_{g(z')}$ est un relev\'e du point $g(z)\in A_{g(z')}$, alors le point $\widetilde z\in \widetilde A_{z'}$ est un relev\'e du point~$z\in A_{z'}$. Enfin, d'apr\`es les \'egalit\'es~\eqref{e.commutation}, si $(p_k)_{k\geq 0}$ est la suite d'entiers telle que $T_{g(z')}^{-p_k}\circ \widetilde f_{I,g(z')}^{n_k}(\widetilde g(\widetilde z))$ converge vers $\widetilde g(\widetilde z)$, alors la suite  $T_{z'}^{-p_k}\circ \widetilde f_{I,z'}^{n_k}(\widetilde z)$ converge vers $\widetilde z$. D'apr\`es le fait~\ref{f.nombre-rotation-7}, ceci montre que le nombre de rotation $\rho_{g(z')}(g(z))$ est bien d\'efini, et est \'egal \`a $\rho_{z'}(z)$.

Le cas o\`u $g$ renverse l'orientation se traite de m\^eme. Les seules diff\'erences sont que l'on a alors $\widetilde g\circ T_{z'} =T^{-1}_{g(z')}\circ \widetilde g$ et que c'est donc la suite $T_{z'}^{p_k}\circ \widetilde f_{I,z'}^{n_k}(\widetilde z)$ qui converge vers $\widetilde z$.
\end{proof}

\bigskip

\begin{proof}[Preuve de l'assertion~(2).]
Ce r\'esultat d\'ecoule de la propri\'et\'e (P1), du lemme de Franks~\ref{theo.lemme-de-Franks}, et du th\'eor\`eme ergodique de Birkhoff. Le point important de la preuve est le suivant~:  \'etant donn\'e $z'\in\mathrm{Fix}(f)$, le lemme de Franks permet de passer d'une propri\'et\'e concernant les nombres de rotation des points fixes de $f$ autour de $z'$ (la fonction $z\mapsto \mathrm{Enlace}_I(z,z')$ est born\'ee sur $\mathrm{Fix}(f)\setminus\{z'\}$) \`a une propri\'et\'e concernant les nombres de rotation de toutes les orbites autour de $z'$ (la fonction  $z\mapsto \mathrm{Enlace}_I(z,z')$ est int\'egrable sur $\mathrm{Rec}^+(f)\setminus\{z'\}$).

Fixons une mesure de probabilit\'e $\mu\in{\mathcal{M}}(f)$ et un point $z'\in\mathrm{Fix}(f)$. Nous appellerons \emph{disque libre} tout disque topologique ouvert disjoint de son image par $f$.  La fonction $\rho_{z'}$ est bien \'evidemment d\'efinie sur  $ \mathrm{Fix}(f)\setminus\{z'\}$ (voir le fait~\ref{f.nombre-rotation-1}). Il reste \`a prouver qu'elle est d\'efinie $\mu$-presque s\^urement sur $ \mathrm{Rec}^+(f)\setminus\mathrm{Fix}(f)$. Comme $\RR^2\setminus\mathrm{Fix}(f)$ est recouvert par les disques libres, il suffit de prouver que, si $U$ est un disque libre, alors $\rho_{z'}$ est d\'efinie $\mu$-presque partout sur $\mathrm{Rec}^+(f)\cap U$. 

Fixons donc un disque libre $U$. Puisque $U$ est ouvert, on peut consid\'erer le temps de premier retour $\tau_U(z)$ dans $U$ de tout point $z\in \mathrm{Rec}^+(f)\cap U$ et d\'efinir une application de premier retour 
\begin{align*}
F_U~: \mathrm{Rec}^+(f)\cap U&\to \mathrm{Rec}^+(f)\cap U\\
z&\mapsto f^{\tau_U(z)}(z).
\end{align*}
Pour tout $z\in \mathrm{Rec}^+(f)\cap U$, on peut alors d\'efinir la quantit\'e
$$\alpha_{U,z'}(z)=\sum_{0\leq \ell<\tau_U(z)} \mathrm{Enlace}_I(f^\ell(z),z')+\int_{\gamma_{U,z,z'}} d\theta,$$
o\`u on choisit un chemin $\gamma_{U,z}$ dans $U$ qui joint $F_U(z)$ \`a $z$ et o\`u on pose $\gamma_{U,z,z'}(t)=\gamma_{U,z}(t)-z'$. Puisque $U$ est simplement connexe, la quantit\'e pr\'ec\'edente ne d\'epend pas du choix de l'arc $\gamma_{U,z}$. 

\begin{remark}
\label{r.interpretation-alpha}
Pour tout couple de points distincts $w,w'$, notons comme dans l'introduction $\gamma_{I,w,w'}:[0,1]\to\RR^2\setminus\{0\}$ le chemin d\'efini par $\gamma_{I,w,w'}(t):=f_t(w)-f_t(w')$. En mettant bout \`a bout les d\'efinitions des quantit\'es  $\alpha_{U,z'}(z)$ et $\mathrm{Enlace}_I(w,z')$, on voit alors que $\alpha_{U,z'}(z)$ est \'egal \`a l'int\'egrale de la 1-forme d'angle polaire $d\theta$ le long du lacet
$$\Gamma_{I,U,z,z'}:=\gamma_{I,z,z'}\,\cdot\,\gamma_{I,f(z),z'}\,\cdot\,\cdots\,\cdot\,\gamma_{I,f^{\tau_U(z)-1}(z),z'}\,\cdot\,\gamma_{U,z,z'}.$$
Ceci montre en particulier que $\alpha_{U,z'}(z)$ est un entier. On peut \'egalement interpr\'eter la quantit\'e $\alpha_{U,z'}(z)$ de la fa\c con suivante. Fixons une composante connexe $\widetilde U$ de $\pi_{z'}^{-1}(U)$. Il s'agit d'un disque topologique ouvert. Tout point $z\in U$ a un unique relev\'e $\widetilde z\in\widetilde U$. Si $z$ appartient \`a $\mathrm{Rec}^+(f)\cap U$, alors $\alpha_{U,z'}(z)$ est l'entier $p$ tel que $(\widetilde f_{I,z'})^{\tau_U(z)}(\widetilde z)\in T_{z'}^p(\widetilde U)$. 
\end{remark}

Puisque $f$ v\'erifie la propri\'et\'e (P1), il existe un entier $M$ tel que  $\vert\mathrm{Enlace}_I(z,z')\vert\leq M$ pour tout $(z,z')\in(\mathrm{Fix}(f)\times\mathrm{Fix}(f))\setminus\Delta$. On a alors le lemme suivant~:
 
\begin{lem}
\label{lem.borne-ensemble-rotation}
L'image de la fonction $\alpha_{U,z'}/\tau_U$, d\'efinie sur $\mathrm{Rec}^+(f)\cap U$, est contenue dans un intervalle de longueur $2M+4$.
\end{lem}

\begin{proof}[Preuve du lemme] 
Raisonnons par l'absurde. Supposons qu'il existe $z^-$ et $z^+$ dans $ \mathrm{Rec}^+(f)\cap U$, tels que  
$$\frac{\alpha_{U,z'}(z^+)}{\tau_U(z^+)} -\frac{\alpha_{U,z'}(z^-)}{\tau_U(z^-)}>2M+4 .$$
Il existe alors au moins $2M+ 2$ entiers $k$ v\'erifiant 
\begin{equation}
\label{e.double-inegalite}
\frac{\alpha_{U,z'} (z^-)}{\tau_U (z^-)}\leq k\leq \frac{\alpha_{U,z'}(z^+)}{\tau_U (z^+)}.
\end{equation}
Soit $\widetilde U$ une composante connexe de $\pi_{z'}^{-1}(U)$. Notons $\widetilde z^-$ et $\widetilde z^+$ les rel\`evement de $z^-$ et $z^+$ situ\'es dans $\widetilde U$. Par ailleurs d\'efinition des fonctions $\tau_U(z^+)$ et $\alpha_{U,z'}$, on a alors
$$\left(\widetilde f_{I,z'}\right)^{\tau_U (z^-)}(\widetilde z^-)\in T_{z'}^{\alpha_{U,z'} (z^-)}(\widetilde U) \quad\mbox{et}\quad \left(\widetilde f_{I,z'}\right)^{\tau_U (z^+)}(\widetilde z^+)\in T_{z'}^{\alpha_{U,z'} (z^+)}(\widetilde U).$$
En particulier
$$\left(\widetilde f_{I,z'}\right)^{\tau_U (z^-)}(\widetilde U)\cap T_{z'}^{\alpha_{U,z'} (z^-)}(\widetilde U)\neq\emptyset \quad\mbox{et}\quad \left(\widetilde f_{I,z'}\right)^{\tau_U (z^+)}(\widetilde U)\in T_{z'}^{\alpha_{U,z'} (z^+)}(\widetilde U)\neq\emptyset.$$
Si $k$ v\'erifie la double in\'egalit\'e~\eqref{e.double-inegalite}, on peut appliquer le lemme de Franks au rel\`evement $\widetilde f_{I,z'}\circ T_{z'}^{-k}$ de $f\vert_{A_{z'}}$ et montrer l'existence d'un point fixe de $\widetilde f_{I,z'}\circ T_{z'}^{-k}$, c'est-\`a-dire l'existence d'un point fixe $z\in\mathrm{Fix}(f)\setminus\{z'\}$ tel que $\mathrm{Enlace}_I(z,z')=k$. Or par hypoth\`ese, il y a au moins $2M+2$ valeurs de $k$ qui v\'erifient la double in\'egalit\'e~\eqref{e.double-inegalite}. Ceci est en contradiction avec l'in\'egalit\'e $\vert\mathrm{Enlace}_I(z,z')\vert\leq M$.
\end{proof}

D'apr\`es le lemme de Kac, la fonction temps de premier retour $\tau_U$ est int\'egrable sur~$U$. On vient de prouver que la fonction $\alpha_{U,z'}/\tau_U$ est born\'ee sur $U$. On en d\'eduit que $\alpha_{U,z'}$ est \'egalement int\'egrable sur $U$. Le th\'eor\`eme ergodique de Birkhoff nous dit que les moyennes de Birkhoff
$$\frac{1}{m}\sum_{\ell=0}^{m-1} \tau_U(F_U^\ell(z))\quad\mathrm{ et }\quad \frac{1}{m}\sum_{\ell=0}^{m-1} \alpha_{U,z'}(F_U^\ell(z))$$
convergent pour $\mu$-presque tout point $z$ de $U$. Nous noterons $\tau_U^*(z)$ et $\alpha_{U,z'}^*(z)$ les limites de ces sommes de Birkhoff lorsqu'elles existent.

 Pour terminer la preuve, consid\'erons un point $z\in\mathrm{Rec}^+(f)\cap U$.  Si $(n_k)_{k\geq 0}$ est une suite d'entiers telle que $f^{n_k}(z)$ tend vers $z$, alors, pour tout $k$ assez grand, le point $f^{n_k}(z)$ est dans le disque $U$, et il existe donc un entier $m_k$ tel que 
$$n_k=\sum_{\ell=0}^{m_k-1} \tau_U(F_U^\ell(z))\quad \mathrm{ et }\quad\sum_{r=0}^{n_k-1}\mathrm{Enlace}(f^{r}(z),z')=\sum_{\ell=0}^{m_k-1} \alpha_{U,z'}(F_U^\ell(z)).$$
Ceci montre que le nombre de rotation $\rho_{z'}(z)$ est bien d\'efini d\`es que les limites $\tau_U^*(z)$ et $\alpha_{U,z'}^*(z)$ existent, et qu'on a alors $\rho_{z'}(z)=\tau_U^*(z)/\alpha_{U,z'}^*(z)$.  
Nous avons ainsi montr\'e que le nombre de rotation $\rho_{z'}(z)$ est bien d\'efini pour $\mu$-presque tout $z$ dans $U$. Comme expliqu\'e plus haut, ceci suffit \`a prouver que $\rho_{z'}(z)$ est bien d\'efini pour $\mu$-presque tout $z$ dans $\RR^2$~: l'assertion~(2) de la proposition~\ref{theo.main-2} est d\'emontr\'ee.
\end{proof}

\begin{remark}
\label{r.nombre-rotation-pas-borne}
Les arguments ci-dessus montrent que, pour tout point $z'\in \mathrm{Fix}(f)$, la fonction $z\mapsto\rho_{z'}(z)$ est born\'ee sur $\mbox{Fix}(f)\setminus \{z'\}$ et localement born\'ee sur $\mathrm{Rec}^+(f)\setminus\mathrm{Fix}(f)$. Dans l'appendice, nous construirons un exemple d'hom\'eomorphisme pour lequel la fonction $z\mapsto\rho_{z'}(z)$ n'est pas globalement born\'ee sur $\mathrm{Rec}^+(f)$.
\end{remark}

\bigskip

\begin{proof}[Preuve de l'assertion~(3).]
On garde les m\^emes notations que dans la preuve de l'assertion~(2). En particulier, si $U$ est un disque libre, et $z'$ un point fixe de $f$, on utilisera les fonctions $\tau_U$, $F_U$ et $\alpha_{U,z'}$ d\'efinies ci-dessus. On utilisera \'egalement la fonction
$$\widehat\alpha_{U,R}:=\mathop{\sup}_{z'\in \mathrm{Fix}(f), \|z'\|\geq R} \vert \alpha_{U,z'}\vert$$
d\'efinie pour tout $R>0$. Le c\oe ur de la preuve de l'assertion~(3) consiste \`a montrer que, pour tout disque libre $U$ et toute mesure $\nu\in\mathcal{M}(f)$, on a
$$\lim_{R\to\infty}\int_{U}  \widehat\alpha_{U,R}(z)\, d\nu(z)=0.$$
Nous en d\'eduirons facilement l'assertion~(3), en utilisant le th\'eor\`eme ergodique de Birkhoff et un petit raisonnement par l'absurde. La limite ci-dessus sera obtenue gr\^ace au th\'eor\`eme de convergence domin\'ee~: la propri\'et\'e~(P2) nous permettra de montrer que $\widehat\alpha_{U,R}(z)$ tend vers $0$, \`a $z$ fix\'e, quand $R$ tend vers l'infini~; ce fait et le lemme~\ref{lem.borne-ensemble-rotation} nous fourniront une domination de la fonction $z\mapsto\widehat\alpha_{U,R}(z)$ par une fonction int\'egrable.

\begin{lem}
\label{lem.alpha-tend-vers-0}
Pour toute mesure $\nu\in\mathcal{M}(f)$, et tout disque libre $U$, on a
$$\lim_{R\to\infty}\;\;\int_{U} \widehat\alpha_{U,R}(z)\, d\nu(z)=0.$$
\end{lem}

\begin{proof}[Preuve du lemme]
On se donne un disque libre $U$, et une mesure $\nu\in\mathcal{M}(f)$.

\medskip

Commen\c cons par montrer que, pour tout $z\in\mathrm{Rec}^+(f)\cap U$ fix\'e, la fonction $z'\mapsto \alpha_{U,z'}(z)$  est localement constante sur $\mathrm{Fix}(f)$. Fixons un point $z\in \mathrm{Rec}^+(f)\cap U$, un chemin $\gamma_{U,z}$ joignant $F_U(z)$ \`a $z$ dans $U$, et un point $z'\in\mathrm{Fix}(f)$. Si $z''\in\mathrm{Fix}(f)$ est assez proche de $z'$, alors on a
$$\left\vert\sum_{0\leq l<\tau_U(z)} \mathrm{Enlace}_I(f^l(z),z')-\sum_{0\leq l<\tau_U(z)} \mathrm{Enlace}_I(f^l(z),z'')\right\vert<\frac{1}{2}$$
et
$$\left\vert\int_{\gamma_{U,z,z'}} d\theta - \int_{\gamma_{U,z,z''}} d\theta \right\vert <\frac{1}{2}.$$
Ces deux in\'egalit\'es impliquent $\left\vert\alpha_{U,z'}(z)-\alpha_{U,z''}(z)\right\vert<1$, et donc $\alpha_{U,z'}(z)=\alpha_{U,z''}(z)$ puisque $\alpha_{U,z'}(z)$ et $\alpha_{U,z''}(z)$ sont des entiers. Nous avons donc bien prouv\'e que la fonction $z'\mapsto \alpha_{U,z'}(z)$ est localement constante.

\medskip

Montrons maintenant que, pour tout $z\in\mathrm{Rec}^+(f)\cap U$ fix\'e, la fonction $z'\mapsto \alpha_{U,z'}(z)$  s'annule au voisinage de l'infini (\emph{i.e.} s'annule sur $\mathrm{Fix}(f)\cap W$, o\`u $W$ est un voisinage de l'infini).
Rappelons que, pour tout point $w\in\RR^2$, nous notons $\gamma_{I,w}:[0,1]\to \RR^2$ le chemin d\'efini par $\gamma_{I,w}(t)=f_t(w)$, et pour tout couple de points distincts $w,w'\in\RR^2$, nous notons $\gamma_{I,w,w'}:[0,1]\to \RR^2$ le chemin d\'efini par $\gamma_{I,w,w'}(t)=f_t(w)-f_t(w')$. Fixons un point $z\in \mathrm{Rec}^+(f)\cap U$, ainsi qu'un chemin $\gamma_{U,z}:[0,1]\to U$ joignant $F_U(z)$ \`a $z$ dans $U$. Pour tout $z'\in \RR^2$, notons $\gamma_{U,z,z'}:[0,1]\to \RR^2$ le chemin d\'efini par $\gamma_{U,z,z'}(t):=\gamma_{U,z}(t)-z'$. Pour tout $z'\in\mathrm{Fix}(f)$, l'entier $\alpha_{U,z'}(z)$ est l'int\'egrale de la 1-forme d'angle polaire $d\theta$ le long du lacet
$$\Gamma_{I,U,z,z'}=\gamma_{I,z,z'}\,\cdot\,\gamma_{I,f(z),z'}\,\cdot\,\cdots\,\cdot\,\gamma_{I,f^{\tau_U(z)-1}(z),z'}\,\cdot\,\gamma_{U,z,z'}$$
(voir la remarque~\ref{r.interpretation-alpha}). Si $z'\in\mathrm{Fix}(f)$ est suffisamment proche de l'infini, le lacet $\Gamma_{I,U,z,z'}$ est librement homotope dans $\RR^2\setminus\{0\}$ au lacet $(\gamma_{I,z'})^{\tau_U(z)}$. Pour le d\'emontrer, il suffit par exemple, de consid\'erer une boule euclidienne $B\subset \RR^2$ centr\'ee \`a l'origine, telle que le lacet $\Gamma_{I,U,z,z'}$ est contenu dans $B$, et de prendre $z'\in\mathrm{Fix}(f)$ suffisamment proche de l'infini pour que le lacet $\gamma_{I,z'}$ soit contenu dans le compl\'ementaire de $B$. Ainsi pour $z'\in\mathrm{Fix}(f)$ suffisamment proche de l'infini, l'entier $\alpha_{U,z'}(z)$ est alors \'egal \`a  l'int\'egrale de $d\theta$ le long de $(\gamma_{I,z'})^{\tau_U(z)}$. Mais cette derni\`ere int\'egrale est par d\'efinition \'egale \`a $\tau_{U}(z)\,\cdot\,\mathrm{Tourne}_I(z')$. Et l'entier $\mathrm{Tourne}_I(z')$ est nul pour tout $z'\in\mathrm{Fix}(f)$ suffisamment proche de l'infini, puisque $f$ v\'erifie (P2) et  $I$ est adapt\'ee. Ceci prouve  que, pour tout $z\in\mathrm{Rec}^+(f)\cap U$ fix\'e, la fonction $z'\mapsto \alpha_{U,z'}(z)$  s'annule au voisinage de l'infini. Autrement dit, pour tout  $z\in\mathrm{Rec}^+(f)\cap U$ fix\'e, la quantit\'e $\widehat\alpha_{U,R}(z)=\mathop{\sup}_{z'\in \mathrm{Fix}(f), \|z'\|\geq R} \vert \alpha_{U,z'}(z)\vert$ est nulle pour $R$ suffisamment grand.
 
 \medskip

Fixons $z_0\in\mathrm{Rec}^+(f)\cap U$. Nous avons montr\'e ci-dessus que la fonction $z'\mapsto \alpha_{U,z'}(z_0)$ est localement constante sur $\mathrm{Fix}(f)$ et nulle en dehors d'un compact~; cette fonction est donc born\'ee. Soit $M_0$ tel que pour tout $z'\in\mathrm{Fix}(f)$, on a
$$\frac{\vert\alpha_{U,z'}(z_0)\vert}{\tau_U(z_0)}\leq M_0.$$ 
D'apr\`es  le lemme~\ref{lem.borne-ensemble-rotation}, on en d\'eduit que pour tout $z\in\mathrm{Rec}^+(f)\cap U$ et tout $z'\in\mathrm{Fix}(f)$, on a  
$$\frac{\vert\alpha_{U,z'}(z)\vert}{\tau_U(z)}\leq M_0+2M+4.$$  
Par cons\'equent, pour tout $z'\in\mathrm{Fix}(f)$, la fonction $z\mapsto \vert\alpha_{U,z'}(z)\vert$ est domin\'ee par la fonction int\'egrable $z\mapsto (M_0+2M+4)\tau_U(z)$. Il en suit imm\'ediatemment que, pour tout $R>0$, la fonction $z\mapsto \widehat\alpha_{U,R}(z)$ est domin\'ee par la fonction int\'egrable $z\mapsto (M_0+2M+4)\tau_U(z)$.

\medskip

Nous avons montr\'e d'une part que, pour tout $z\in\mathrm{Rec}^+(f)\cap U$, la quantit\'e $\widehat\alpha_{U,R}(z)$ est nulle pour $R$ suffisamment grand, et d'autre part que, pour tout $R>0$, la fonction positive $z\mapsto \widehat\alpha_{U,R}(z)$ est domin\'ee par une fonction int\'egrable ind\'ependante de $R$. Nous pouvons donc appliquer le th\'eor\`eme de convergence domin\'ee de Lebesgue et obtenir la limite annonc\'ee. 
\end{proof}

\begin{lem}
\label{c.nombre-rotation-moyen-tend-vers-zero}
Pour toute mesure $\nu\in\mathcal{M}(f)$, tout disque libre $U$, et tout point $z'\in\mathrm{Fix}(f)$, on a
$$\int_{\bigcup_{k\geq 0} f^k(U)}\rho_{z'}(z)\,d\nu(z)=\int_{U} \alpha_{U,z'}(z)\, d\nu(z).$$
\end{lem}

\begin{proof}[Preuve du lemme]
On a 
$$\int_{\bigcup_{k\geq 0} f^k(U)}\rho_{z'}d\nu=\int_{U}\tau_U\rho_{z'}d\nu=\int_{U}\tau_U^*\rho_{z'}d\nu=\int_{U} \alpha_{U,z'}^*d\nu=\int_{U} \alpha_{U,z'}d\nu$$
o\`u $\tau_U^*$ et $\alpha_{U,z'}^*$ sont les limites des moyennes de Birkhoff des fonctions $\tau_U$ et $\alpha_{U,z'}$ (on rappelle que ces limites existent pour $\nu$ presque partout sur $U$~; voir la preuve de l'assertion~(2)). La premi\`ere \'egalit\'e ci-dessus d\'ecoule de l'invariance de la mesure $\nu$ et de la fonction $\rho_{z'}$ sous l'action de $f$. La deuxi\`eme et la derni\`ere d\'ecoulent du th\'eor\`eme ergodique de Birkhoff. Enfin, la troisi\`eme est une cons\'equence de l'\'egalit\'e $\rho_{z'}(z)=\alpha_{U,z'}^*(z)/\tau_U^*(z)$, que nous avons montr\'ee \`a la fin de la preuve de l'assertion~(2). 
\end{proof}

\begin{lem}
\label{c.nombre-rotation-moyen-tend-vers-zero-2}
Soit $(z'_n)$ est une suite de points dans $\mathrm{Fix}(f)$. Supposons que $\|z'_n\|\geq R$ pour tout $n\geq n_0$. Alors, pour toute mesure $\nu\in\mathcal{M}(f)$, tout disque libre $U$,  on a
$$\int_{\bigcup_{k\geq 0} f^k(U)}\mathop{\sup}_{n\geq n_0} \vert\rho_{z'_n}(z)\vert\,d\nu(z)\leq \int_{U} \widehat\alpha_{U,R}(z)\, d\nu(z).$$
\end{lem}

\begin{proof}[Preuve du lemme]
La preuve est similaire \`a celle du lemme pr\'ec\'edent. On a 
$$\int_{\bigcup_{k\geq 0} f^k(U)}\mathop{\sup}_{n\geq n_0} \vert\rho_{z'_n}\vert\,d\nu\quad\mathop{=}^{(a)}\quad\int_{U}\tau_U\mathop{\sup}_{n\geq n_0} \vert\rho_{z'_n}\vert\,d\nu\quad\mathop{=}^{(b)}\quad\int_{U}\tau_U^*\mathop{\sup}_{n\geq n_0} \vert\rho_{z'_n}\vert\,d\nu\mathop{=}^{(c)}\quad\int_{U} \sup_{n\geq n_0}\vert\alpha_{U,z'_n}^*\vert\,d\nu$$
$$\quad\quad\quad\quad\quad\mathop{\leq}^{(d)}\quad\int_{U} \left(\sup_{n\geq n_0}\vert\alpha_{U,z'_n}\vert\right)^*\,d\nu\quad\mathop{\leq}^{(e)}\quad\int_{U} \widehat\alpha_{U,R}^*d\nu\quad\mathop{=}^{(f)}\quad\int_{U} \widehat\alpha_{U,R}d\nu$$
o\`u $\phi^*$ d\'esigne la limite des moyennes de Birkhoff d'une fonction int\'egrable $\phi$. L'\'egalit\'e~(a) ci-dessus d\'ecoule de l'invariance de la mesure $\nu$ et de la fonction $\rho_{z'}$ sous l'action de $f$. Les \'egalit\'es~(b) et~(f) d\'ecoulent du th\'eor\`eme ergodique de Birkhoff. L'\'egalit\'e~(c) est une cons\'equence de l'\'egalit\'e $\rho_{z'}(z)=\alpha_{U,z'}^*(z)/\tau_U^*(z)$. L'in\'egalit\'e~(d) exprime le fait qu'une moyenne de Birkhoff du supremum d'une famille de fonction est sup\'erieure au supremum des moyennes de Birkhoff de ces fonctions. Enfin, l'in\'egalit\'e~(e) d\'ecoule de la d\'efinition de la fonction $\widehat\alpha_{U,R}$.
\end{proof}

Pour terminer la preuve de l'assertion (3) de la proposition~\ref{theo.main-2}, nous raisonnons maintenant par l'absurde. Supposons que l'assertion~(3) est fausse. Il existe alors deux constantes $a>0$ et $b>0$, ainsi qu'une suite de points $(z'_n)_{n\geq 0}$ dans $\mathrm{Fix}(f)$ qui tend vers l'infini telles que, si on note 
$$A_n=\{z\in \mathrm{Rec}^+(f)\;,\;\vert\rho_{z'_n}(z)\vert\geq a\},$$
alors $\mu(A_n)\geq b$ pour tout $n$.
L'ensemble $A=\bigcap_{n_0\geq 0}(\bigcup_{n\geq n_0} A_n)$ est alors invariant par $f$, et v\'erifie $\mu(A)\geq b$.  De plus, pour tout $z\in A$ et tout $n_0\geq 0$, on a $\sup_{n\geq n_0} \vert\rho_{z'_n}(z)\vert\geq a$. Posons $\nu=\mu\vert_A$. Puisque $I$ est adapt\'ee, pour tout point fixe $z$, on a $\mathrm{Enlace}_I(z,z'_n)=0$, si $n$ est assez grand. Ainsi $A$ ne contient aucun point fixe. Il existe donc un disque libre $U$ tel que $\nu(U)>0$ et 
 $$\int_{\cup_{k\geq 0} f^k(U)}\sup_{n\geq n_0}\vert\rho_{z'_n}\vert\,d\nu\geq a\nu\left(\bigcup_{k\geq 0} f^k(U)\right)\geq a\nu(U)$$
pour tout $n_0$. Cette in\'egalit\'e, le lemme~\ref{c.nombre-rotation-moyen-tend-vers-zero-2}, et le fait que la suite $(z'_n)$ tend vers l'infini, montrent que, pour tout $R\geq 0$, on a 
$$\int_U \widehat\alpha_{U,R}\,d\nu\geq a\nu(U).$$
Ceci contredit le lemme~\ref{lem.alpha-tend-vers-0}. Cette contradiction ach\`eve la preuve de l'assertion~(3).
\end{proof}

\bigskip
\begin{proof}[Preuve de l'assertion~(4).] 
Le point cl\'e de la preuve est le lemme suivant.

\begin{lem}
\label{lem.propriete-diffeo-C1}
Notons $\mathbb{S}^2$ la sph\`ere unit\'e de $\mathbb{R}^3$, et consid\'erons un $C^1$-diff\'eomorphisme $g$ de $\mathbb{S}^2$, qui fixe un point not\'e $\infty$. Pour tout point $z\in\mathbb{S}^2$, distinct de $\infty$ et de son point antipodal, notons $\gamma_z$ l'unique grand cercle qui passe par $\infty$ et $z$, et notons $\gamma^-_z$ (respectivement $\gamma^+_z$) le petit (respectivement grand) arc de $\gamma_z$ joignant $\infty$ \`a $z$. Alors, il existe un voisinage point\'e $W$ de $\infty$ dans $\mathbb{S}^2$, tel que pour tout point fixe $z\in W$ on a $g(\gamma^-_z)\cap \gamma^+_z=\{z,\infty\}$.
\end{lem}

\begin{proof}[Preuve du lemme] 
Raisonnons par l'absurde. Dire que la conclusion n'est pas v\'erifi\'ee signifie qu'il existe une suite de points fixes $(z_n)_{n\geq 0}$ tous distincts de $\infty$ qui converge vers $\infty$ et une suite de points $(z'_n)_{n\geq 0}$ tels que $z'_n\in \gamma^-_{z_n}\setminus\{z_n,\infty\}$ et  $g(z'_n)\in\gamma^+_{z_n}$.  

Notons $Dg(z)$ la diff\'erentielle de $g$ en un point $z\in\SS^2$. Notons $v_n\in T_\infty{\mathbb S}^2$ le vecteur unitaire tangent \`a l'arc $\gamma^-_{z_n}$ au point $\infty$ qui pointe vers $z_n$.  Quitte \`a extraire une sous-suite, on peut toujours supposer que la suite $(v_n)_{n\geq 0}$ est convergente~; on note $v$ sa limite. L'application $g$ est diff\'erentiable au point $\infty$, et la quantit\'e $\|Dg(\infty).v_n-v_n\|$ doit tendre vers $0$ quand $n$ tend vers l'infini, puisque les deux extr\'emit\'es $\infty$ et $z_n$ de l'arc $\gamma^-_{z_n}$ sont des points fixes de $g$, et puisque la longueur de l'arc $\gamma^-_{z_n}$ tend vers $0$. Par suite, le vecteur $v$ est dans l'espace propre associ\'e \`a la valeur propre $1$ de la diff\'erentielle de $g$ au point $\infty$.

Notons maintenant $u_n\in T_{z_n}{\mathbb S}^2$ le vecteur unitaire tangent \`a l'arc $\gamma^-_{z_n}$  au point $z_n$ qui pointe vers $z_n'$. On remarque que la suite $(z_n,u_n)$ converge vers $(\infty,-v)$ dans $T{\mathbb S}^2$. On distingue alors deux cas~: 

\smallskip

\noindent \textit{Premier cas.} Quitte \`a extraire une sous-suite, la longueur du sous-arc de $\gamma_{z_n}^+$ joignant $\infty$ \`a $g(z_n')$ tend vers $0$.

\smallskip

\noindent \textit{Second cas.}  Quitte \`a extraire une sous-suite, la longueur du sous-arc de $\gamma_{z_n}^+$ joignant $z_n$ \`a $g(z_n')$ tend vers $0$.

\smallskip

\noindent Soit $D$ un disque dans $\SS^2$ centr\'e en $\infty$. Dans le premier cas, comme les points $z'_n$ et $g(z'_n)$ sont situ\'es de part et d'autre du point $\infty$ sur $\gamma_{z_n}\cap D$, on voit que l'angle entre $v_n$ et $Dg(\infty).v_n$ doit tendre vers $\pi$ quand $n\to \infty$. Par cons\'equent, le vecteur $-v$ doit \^etre dans l'espace propre associ\'e \`a une valeur propre n\'egative ou nulle de la diff\'erentielle de $g$ au point $\infty$. Dans le second cas, comme les points $z'_n$ et $g(z'_n)$ sont situ\'es de part et d'autre du point $z_n$ sur $\gamma_{z_n}\cap D$, on voit que l'angle entre $u_n$ et $Dg(z_n).u_n$ doit tendre vers $\pi$ quand $n$ tend vers l'infini. Comme la diff\'erentielle de $g$ est continue, ceci implique \`a nouveau que le vecteur $v$ doit \^etre dans l'espace propre associ\'e \`a une valeur propre n\'egative ou nulle de la diff\'erentielle de $g$ au point $\infty$. 

Nous avons ainsi montr\'e que le vecteur $v$ est \`a la fois dans l'espace propre associ\'e \`a la valeur propre $1$ et dans un espace propre associ\'e \`a une valeur propre n\'egative ou nulle de la diff\'erentielle de $g$ au point $\infty$, ce qui constitue la contradiction recherch\'ee.
\end{proof}

Supposons donc que $f$ s'\'etend en un $C^1$-diff\'eomorphisme $\bar f$ de $\mathbb{S}^2\simeq\mathbb{R}^2\sqcup\{\infty\}$. Fixons $q\geq 1$, et appliquons le lemme~\ref{lem.propriete-diffeo-C1} au diff\'eomorphisme $\bar f^q$. Si $z'$ est un point fixe de $f^q$ suffisamment proche de $\infty$, le lemme nous fournit deux arcs $\gamma^-$ et $\gamma^+$ dans l'anneau $A_{z'}=\mathbb{S}^2\setminus\{\infty,z'\}$, joignant un bout \`a l'autre de l'anneau, disjoints, et tels que $f^q(\gamma^-)\cap \gamma^+=\emptyset$. Rappelons que $\pi_{z'}:\widetilde A_{z'}\to A_{z'}$ d\'esigne le rev\^etement universel de l'anneau $A_{z'}$, que $\widetilde f_{I,z'}$ d\'esigne le relev\'e de $f\vert_{A_{z'}}$ \`a $\widetilde A_{z'}$ tel que $I$ se rel\`eve en une isotopie de l'identit\'e \`a $\widetilde f_{I,z'}$, et que $T_{z'}$ d\'esigne le g\'en\'erateur du groupe des automorphismes du rev\^etement  $\pi_{z'}:\widetilde A_{z'}\to A_{z'}$ naturellement d\'efini par l'orientation du bord d'un disque centr\'e en $z'$. Choisissons un relev\'e $\widetilde \gamma^-$ de l'arc $\gamma^-$ dans $\widetilde A_{z'}$ (autrement dit, $\widetilde \gamma^-$  est une composante connexe de $\pi_{z'}^{-1}(\gamma^-)$). Alors l'arc $\widetilde f_{I,z'}^q(\widetilde \gamma^-)$ ne rencontre aucune composante connexe de $\pi^{-1}_{z'}(\gamma^+)$, et d'autre part, il rencontre au plus un translat\'e $T_{z'}^k(\widetilde \gamma^-)$ de $\widetilde \gamma^-$. Puisque $\widetilde f_{I,z'}$ a des points fixes (les relev\'es des points fixes proches de l'infini), l'arc $\widetilde f_{I,z'}^q(\widetilde \gamma^-)$ ne peut pas \^etre \`a droite de l'arc $T_{z'}(\widetilde \gamma^-)$. Il est donc situ\'e \`a gauche de l'arc $T_{z'}^2(\widetilde \gamma^-)$. De m\^eme il est situ\'e \`a droite de l'arc $T_{z'}^{-2}(\widetilde \gamma^-)$. Comme $\widetilde f_{I,z'}^q$ et $T_{z'}$ commutent, ceci implique \'evidemment que l'arc $\widetilde f_{I,z'}^q(T_{z'}(\widetilde \gamma^-))$ est \`a gauche de l'arc $T_{z'}^3(\widetilde \gamma^-)$ et \`a droite de l'arc $T_{z'}^{-1}(\widetilde \gamma^-)$. Consid\'erons maintenant un point $z\in\mathrm{Rec}^+(f)\setminus\{z'\}$ tel que le nombre de rotation $\rho_{z'}(z)$ est bien d\'efini. Il existe un (unique) relev\'e $\widetilde z$ de ce point dans $\widetilde A_{z'}$ qui est situ\'e entre les arcs $\widetilde \gamma^-$ et $T_{z'}(\widetilde \gamma^-)$. De ce qui pr\'ec\`ede, on d\'eduit donc que le point $\widetilde f_{I,z'}^q(\widetilde z)$ est situ\'e entre $T_{z'}^{-2}(\widetilde \gamma^-)$ et $T_{z'}^3(\widetilde \gamma^-)$. Par r\'ecurrence, on en d\'eduit que le point $\widetilde f_{I,z'}^{nq}(\widetilde z)$ est situ\'e  entre $T_{z'}^{-2n}(\widetilde \gamma^-)$ et $T_{z'}^{3n}(\widetilde \gamma^-)$ pour tout $n\geq 0$. En utilisant la caract\'erisation du nombre de rotation $\rho_{z'}(z)$ via le relev\'e $\widetilde f_{I,z'}$ (fait~\ref{f.nombre-rotation-7}), on en d\'eduit que $\vert \rho_{z'}(z)\vert\leq 3/q$. Comme on peut choisir l'entier $q$ aussi grand que l'on veut, ceci d\'emontre l'assertion~(4).
\end{proof}

\bigskip

\begin{proof}[Preuve de l'assertion~(5).] 
Soit $\mu$ une mesure bor\'elienne finie sur $\mathbb{R}^2$, invariante par $f$, dont le support n'est pas contenu dans l'ensemble des points fixes de $f$.

\bigskip

La preuve de l'assertion~(5) repose sur le th\'eor\`eme~\ref{theo.Brouwer-feuillete-equivariant}. Ce th\'eor\`eme fournira un feuilletage singulier orient\'e $\mathcal{F}$ sur $\mathbb{S}^2$ tel que, pour tout $z\in\mathbb{S}^2$ qui n'est pas une singularit\'e de $\mathcal{F}$, le chemin $\gamma_{I,z}:t\mapsto f_t(z)$ est homotope \`a un chemin positivement transverse \`a $\mathcal{F}$. De mani\`ere informelle, l'existence d'un tel feuilletage signifie  que, pour toute singularit\'e $z'$ de $\mathcal F$, le chemin $\gamma_{I,z}$ tourne autour de $z'$, et le sens de rotation ne d\'epend pas du point $z$. Ceci nous permettra de trouver un point $z'\in\mathrm{Fix}(f)$, et un disque libre $V$ tel que $\mu(V)>0$ et tel que $\alpha_{V,z'}(z)$ est strictement positif pour tout point $z\in (V\cap \mathrm{Rec}^+(f))\setminus X$, ou strictement n\'egatif pour tout $z\in (V\cap \mathrm{Rec}^+(f))\setminus X$. Nous conclurons alors via le lemme~\ref{c.nombre-rotation-moyen-tend-vers-zero}. 

Dans les faits, la d\'emonstration sera assez technique, et cela pour deux raisons au moins. La premi\`ere est que nous n'avons fait pratiquement aucune hypoth\`ese sur la mesure $\mu$. La seule chose que nous savons, c'est que le support de $\mu$ n'est pas contenu dans l'ensemble des points fixes de $\mathcal{F}$. La preuve aurait \'et\'e nettement simplifi\'ee si nous avions suppos\'e par exemple que $\mu$ chargeait tous les ouverts de $\mathbb{S}^2$, ou que $\mu$ \'etait ergodique. La seconde est que le th\'eor\`eme~\ref{theo.Brouwer-feuillete-equivariant}, que nous appliquerons sur $\mathbb{R}^2\setminus\{z'\}$ pour un certain  $z'\in\mathrm{Fix}(f)$, va nous fournir une isotopie de l'identit\'e \`a $f$ sur $\mathbb{R}^2\setminus(\{z'\}\cup X)$ o\`u $X$ est un ferm\'e de $\mathrm{Fix}(f)\setminus\{z'\}$ dont nous ne savons rien. Il se pourrait par exemple que $X$ disconnecte $\mathbb{R}^2\setminus\{z'\}$, ce qui nous causera quelques difficult\'es. 

\bigskip

Avant de commencer la preuve de l'assertion~(5) proprement dite, nous avons besoin de faire quelques rappels sur les feuilletages de la sph\`ere $\mathbb{S}^2$. Soit $\mathcal{F}$ un feuilletage singulier orient\'e de la sph\`ere $\mathbb{S}^2$. Comme $\mathcal{F}$ est orient\'e, on peut parler des ensembles $\alpha$-limite et $\omega$-limite d'une feuille $\lambda$~; nous noterons  $\alpha(\lambda)$ et $\omega(\lambda)$ ces ensembles. Consid\'erons une feuille $\lambda$ de $\mathcal{F}$ qui n'est pas ferm\'ee. Alors $\omega(\lambda)$ est une partie de $\mathbb{S}^2$ non-vide, ferm\'ee, connexe, disjointe de $\lambda$, et satur\'ee par $\mathcal{F}$ (ceci signifie que toute feuille de $\mathcal{F}$ qui rencontre $\omega(\lambda)$ est contenue dans $\omega(\lambda)$). Nous noterons $\widehat \omega(\lambda)$ le compl\'ementaire de la composante connexe de $\mathbb{S}^2\setminus \omega(\lambda)$ qui contient $\lambda$. Les arguments de la preuve du th\'eor\`eme de Poincar\'e-Bendixon montrent qu'on a les propri\'et\'es suivantes~:

\medskip

\noindent \textbf{(F1)} L'ensemble  $\widehat \omega(\lambda)$ est une partie de $\mathbb{S}^2$ compacte, connexe, pleine (c'est-\`a-dire de compl\'ementaire connexe), satur\'ee par $\mathcal{F}$, qui contient au moins une singularit\'e de $\mathcal{F}$. La feuille $\lambda$ s'accumule sur chaque point de la fronti\`ere de $\widehat \omega(\lambda)$. 

\medskip

\noindent \textbf{(F2)} Si $\widehat\omega(\lambda)$ n'est pas r\'eduit \`a une singularit\'e de $\mathcal{F}$, alors on a l'alternative suivante~: soit tout arc positivement transverse \`a $\mathcal{F}$ issu d'un point de $\widehat \omega(\lambda)$ est inclus dans $\widehat \omega(\lambda)$, soit tout arc positivement transverse \`a $\mathcal{F}$ aboutissant \`a un point de  $\widehat \omega(\lambda)$ est inclus dans $\widehat \omega(\lambda)$. 

\medskip

\noindent Supposons maintenant que les singularit\'es de $\mathcal{F}$ sont des points fixes de $f$ et que, pour tout point $z\in\mathbb{S}^2\setminus \mathrm{Sing}(\mathcal{F})$, le chemin $\gamma_{I,z}:t\mapsto f_t(z)$ soit homotope, \`a extr\'emit\'es fix\'ees, relativement aux singularit\'es de $\mathcal{F}$, \`a un arc positivement transverse \`a $\mathcal{F}$. De la propri\'et\'e~(F2), on d\'eduit~:

\medskip

\noindent \textbf{(F3)} L'ensemble $\widehat \omega(\lambda)$  est positivement ou n\'egativement invariant par $f$. 

\medskip

\noindent On d\'efinit l'ensemble $\widehat\alpha(\lambda)$ de fa\c con similaire \`a $\widehat\omega(\lambda)$, et les propri\'et\'es (F1), (F2), (F3) restent bien s\^ur valables si on y remplace $\widehat\omega(\lambda)$ par  $\widehat\alpha(\lambda)$.  

\bigskip

Nous d\'ebutons maintenant la preuve de l'assertion~(5). Le support de la mesure $\mu$ n'est pas inclus dans l'ensemble des points fixes de $f$. Comme $\mu$ presque tout point est positivement r\'ecurrent pour $f$, on peut donc trouver un point $z_0\in (\mathrm{Supp}(\mu)\cap \mathrm{Rec}^+(f))\setminus\mathrm{Fix}(f)$. Consid\'erons un point $z'\in\mathrm{Fix}(f)$. Quitte \`a remplacer $I$ par une autre isotopie adapt\'ee, nous pouvons supposer que $I$ fixe le point $z'$. Appliquons alors le th\'eor\`eme~\ref{theo.Brouwer-feuillete-equivariant} \`a l'anneau $M:=A_{z'}=\mathbb{R}^2\setminus\{z'\}$ et \`a  l'isotopie $I\vert_{A_{z'}}$. Ce th\'eor\`eme nous fournit une partie ferm\'ee $X$ de $\mathrm{Fix}(f)\setminus\{z'\}$, une isotopie $I'=(f'_t)_{t\in [0,1]}$ joignant l'identit\'e \`a $f\vert_{A_{z'}\setminus X}$ dans $\mathrm{Homeo}(A_{z'}\setminus X)$ et un feuilletage (non-singulier) $\mathcal{F}$ sur $A_{z'}\setminus X$. On posera alors $\widehat{X}=X\cup\{z',\infty\}$, et on verra $\mathcal{F}$ comme un feuilletage singulier sur $\mathbb{S}^2=A_{z'}\sqcup\{z',\infty\}$, ayant $\widehat{X}$ comme ensemble de singularit\'es. Remarquons que $\infty$ n'est pas un point isol\'e de $\widehat{X}$. En effet, dans le cas contraire, la trajectoire $\gamma_{I',z}$ d'un point fixe proche de l'infini serait homotope \`a z\'ero dans $A_{z'}\setminus X$. 

Le point $z_0$ n'appartient pas \`a $\widehat X$, puisqu'il n'est pas dans $\mathrm{Fix}(f)$. Comme $z_0$ est positivement r\'ecurrent pour $f$, il appartient \`a un lacet\footnote{Ce lacet est obtenu en choisissant un entier $q$ tel que $f^q(z_0)$ est tr\`es proche de $z_0$, et en perturbant le lacet  obtenu en concat\'enant des chemins positivement transverses \`a $\mathcal{F}$ homotopes \`a $\gamma_{I,z_0}, \gamma_{I,f(z_0)},Ê\dots,\gamma_{I,f^{q-1}(z_0)}$ et un petit arc joignant  $f^q(z_0)$ \`a $z_0$. Voir~\cite{LeCalvez2005} pour les d\'etails.} qui est positivement transverse \`a $\mathcal{F}$. On en d\'eduit d'une part que la feuille $\lambda_0$ de $\mathcal{F}$ qui contient $z_0$ n'est pas ferm\'ee mais \'egalement que $\widehat \alpha(\lambda_0)$ et $\widehat \omega(\lambda_0)$ sont disjoints. Notons $A=\mathbb{S}^2\setminus (\widehat \alpha(\lambda_0)\sqcup\widehat \omega(\lambda_0))$. La propri\'et\'e~(F1) montre que $A$ est un anneau ouvert, et $\lambda_0$ va d'un bout \`a l'autre de cet anneau. Puisque $\widehat \alpha(\lambda_0)$ et $\widehat \omega(\lambda_0)$  sont positivement ou n\'egativement invariants par $f$ (propri\'et\'e~(F3)), on sait que la mesure $\nu:=\mu\vert_A$ est invariante par $f$. Nous allons envisager cinq cas suivant la position des points $\infty$ et $z'$~:

\medskip
\begin{tabular}{lll}
1. $\infty\in\widehat\omega(\lambda_0)$ et $z'\in\widehat\alpha(\lambda_0)$~; & &1bis. $\infty\in\widehat\alpha(\lambda_0)$ et $z'\in\widehat\omega(\lambda_0)$~;\\
2.  $\infty\in\widehat\omega(\lambda_0)$ et $z'\not\in\widehat\alpha(\lambda_0)$~; & & 2bis. $\infty\in\widehat\alpha(\lambda_0)$ et $z'\not\in\widehat\omega(\lambda_0)$~;\\
3. $\infty\not\in(\widehat \alpha(\lambda_0)\sqcup\widehat \omega(\lambda_0))$.
\end{tabular}

\bigskip

Commen\c cons par le cas le plus simple~: celui o\`u $\infty\in\widehat\omega(\lambda_0)$ et $z'\in\widehat\alpha(\lambda_0)$.  Nous allons construire un disque libre $V$ tel que $\mu(V)>0$ et tel que $\alpha_{V,z'}(z)>0$ pour tout $z\in V$. 

Rappelons que $\pi_{z'}:\widetilde A_{z'}\to A_{z'}$ d\'esigne le rev\^etement universel de l'anneau $A_{z'}=\mathbb{S}^2\setminus\{\infty,z'\}$. Notons $\widetilde{\mathcal{F}}$ le feuilletage singulier orient\'e de $\widetilde A_{z'}$ obtenu en relevant $\mathcal{F}$. Remarquons que $A$ est un sous-anneau essentiel de l'anneau $A_{z'}$~; par suite, $\widetilde A:=\pi_{z'}^{-1}(A)$ est une bande dans $\widetilde A_{z'}$. Fixons alors un relev\'e $\widetilde z_0$ de $z_0$ dans $\widetilde A_{z'}$, et notons $\widetilde\lambda_0$ la feuille de $\widetilde{\mathcal{F}}$ passant par $\widetilde z_0$. Remarquons que $\widetilde\lambda_0$ est une droite topologique orient\'ee proprement plong\'ee dans la bande $\widetilde A$. En effet, c'est un relev\'e de la feuille $\lambda_0$ de $\mathcal{F}$ dont on sait qu'elle va d'un bout \`a l'autre de l'anneau $A$. Cela a un sens de dire qu'un point $\widetilde z\in\widetilde A$ est situ\'e \emph{\`a gauche} (respectivement \emph{\`a droite}) de $\widetilde\lambda_0$. Rappelons maintenant que $\widetilde f_{I,z'}$ d\'esigne le rel\`evement de $f$ \`a $\widetilde A_{z'}$ associ\'e \`a l'isotopie $I$. Ici, comme $I$ fixe $z'$, elle se rel\`eve en une isotopie qui joint $\mathrm{Id}\vert_{\widetilde A_{z'}}$ \`a $\widetilde f_{I,z'}$. 

\begin{lem}
\label{l.va vers-la-droite}
Consid\'erons un point $z\in A\setminus X$ et un relev\'e $\widetilde z$ de $z$ dans $\widetilde A$. Si $z$ est positivement r\'ecurrent pour $f$, alors le point $f(z)$ est aussi dans~$A$. Si, de plus, $\widetilde z$ est \`a droite au sens large de la feuille $\widetilde\lambda_0$, alors son image $\widetilde f_{I,z'}(\widetilde z)$ est \`a droite au sens strict de $\widetilde\lambda_0$. 
\end{lem}

\begin{proof}[Preuve du lemme]
Les propri\'et\'es d'invariance des ensembles $\widehat\alpha(\lambda_0)$ et $\widehat\omega(\lambda_0)$ (propri\'et\'e~(F3) et son analogue pour $\widehat\alpha(\lambda_0)$) montrent que, si $z$ est positivement r\'ecurrent pour $f$, alors l'orbite de $f$ doit \^etre contenue dans $A$. Ceci montre la premi\`ere affirmation du lemme.

Passons \`a la preuve de la seconde affirmation. D'apr\`es les assertions~(3) et (4) du th\'eor\`eme~ 4.2, le chemin $\gamma_{I,z}$ est homotope dans $A_{z'}$ \`a un chemin $\gamma_z$ positivement transverse au feuilletage $\mathcal{F}$ 
et contenue dans $A$ car $\lambda_{0}$ s'accumule sur tous les points de la fronti\`ere de $\widetilde{\alpha}(\lambda_{0})$ et $\widetilde{\omega}(\lambda_{0})$. 
Par ailleurs, l'isotopie $I$ se rel\`eve dans $\mathrm{Homeo}(\widetilde A_{z'})$ en une isotopie qui joint l'identit\'e \`a $\widetilde f_{I,z'}$ (par d\'efinition de $\widetilde f_{I,z'}$ et parce que $I$ fixe $z'$). Ainsi, le chemin $\gamma_{I,z}$ se rel\`eve en un chemin joignant $\widetilde z$ \`a $\widetilde f_{I,z'}(\widetilde z)$. Le chemin $\gamma_z$ \'etant homotope \`a $\gamma_{I,z}$, il se rel\`eve \'egalement en un chemin $\widetilde\gamma_z$ joignant $\widetilde z$ \`a $\widetilde f_{I,z'}(\widetilde z)$. Bien s\^ur, ce chemin est positivement transverse au feuilletage~$\widetilde{\mathcal{F}}$ et il est contenu dans $\widetilde{A}$. En particulier, il ne peut rencontrer la feuille $\widetilde\lambda_0$ qu'en la traversant de la gauche vers la droite ce qui termine la preuve de la seconde affirmation du lemme.
\end{proof}

Comme le point $\widetilde z_0$ est situ\'e sur la feuille $\widetilde\lambda_0$, le lemme~\ref{l.va vers-la-droite} implique que le point $\widetilde f_{I,z'}(\widetilde z_0)$ est situ\'e strictement \`a droite de $\widetilde\lambda_0$, et que le point $\widetilde f_{I,z'}^{-1}(\widetilde z_0)$ est situ\'e strictement \`a gauche de $\widetilde\lambda_0$. Consid\'erons un disque topologique ouvert $\widetilde U\subset\widetilde A_{z'}$ contenant le point $\widetilde z_0$. Quitte \`a choisir $\widetilde U$ suffisamment petit, nous supposerons que $U:=\pi_{z'}(\widetilde U)$ est un disque libre, que les disques topologiques $U$, $f(U)$ et $f^{-1}(U)$ sont tous les trois contenus dans $A$, que le disque $\widetilde U$ est situ\'e \`a droite au sens strict de la feuille $T_{z'}^{-1}(\widetilde\lambda_0)$ et \`a gauche au sens strict de la feuille $T_{z'}(\widetilde\lambda_0)$, que le disque ouvert $\widetilde f_{I,z'}(\widetilde U)$ est situ\'e \`a droite au sens strict de la feuille $\widetilde\lambda_0$, et que  le disque ouvert $\widetilde f_{I,z'}^{-1}(\widetilde U)$ est situ\'e  \`a gauche au sens strict de $\widetilde\lambda_0$.

Nous allons montrer que la fonction $\alpha_{U,z'}$ est strictement positive sur $U$. Pour ce faire, nous divisons $U$ en deux parties~: soit $\widetilde U_G$ (respectivement $\widetilde U_D$) l'ensemble des points de $\widetilde U$ situ\'es \`a gauche (respectivement \`a droite) au sens large de la feuille $\widetilde\lambda_0$. Nous notons $U_G:=\pi_{z'}(\widetilde U_G)$ et $U_D:=\pi_{z'}(\widetilde U_D)$. Consid\'erons un point $z\in \mathrm{Rec}^+(f)\cap U$, et notons $\widetilde z$ le relev\'e de $z$ situ\'e dans $\widetilde U$. Pour simplifier les notations, posons $q:=\tau_U(z)>1$. Il existe alors un unique entier $p\in\ZZ$, tel que $\widetilde f_{I,z'}^q(\widetilde z)\in T_{z'}^p(\widetilde U)$. D'apr\`es la remarque~\ref{r.interpretation-alpha}, on a $\alpha_{U,z'}(z)=p$. Nous devons donc montrer que l'entier $p$ est strictement positif. Nous distinguons deux cas~:

\medskip

\noindent \textit{Cas~A. Le point $\widetilde f_{I,z'}^q(\widetilde z)$ est dans $T_{z'}^p(\widetilde U_G)$.} Par construction,  $\widetilde U_G$ est situ\'e \`a gauche de  $\widetilde\lambda_0$, donc $T_{z'}^k(\widetilde U_G)$ est situ\'e \`a gauche de $\widetilde\lambda_0$ pour tout $k\leq 0$. Mais $\widetilde f_{I,z'}(\widetilde U)$ est \`a droite au sens strict de $\widetilde\lambda_0$, ce qui implique, en utilisant le lemme~\ref{l.va vers-la-droite}, que le point $\widetilde f_{I,z'}^\ell(\widetilde z)$ est \`a droite au sens strict de $\widetilde\lambda_0$ pour tout $\ell\geq 1$. En particulier, $\widetilde f_{I,z'}^q(\widetilde z)$ est \`a droite au sens strict de $\widetilde\lambda_0$. Par cons\'equent, l'entier $p=\alpha_{U,z'}(z)$ doit \^etre strictement positif.

\medskip

\noindent \textit{Cas~B. Le point $\widetilde f_{I,z'}^q(\widetilde z)$ est dans $T_{z'}^p(\widetilde U_D)$.} Par construction,  $\widetilde U_D$ est situ\'e \`a gauche de $T_{z'}(\widetilde\lambda_0)$, donc  $T_{z'}^k(\widetilde U_D)$ est situ\'e \`a gauche de $\widetilde\lambda_0$ pour tout $k<0$. Mais, comme dans le premier cas, le point $\widetilde f_{I,z'}^\ell(\widetilde z)$ est \`a droite au sens strict de $\widetilde\lambda_0$ pour tout $\ell\geq 1$. Par cons\'equent, l'entier $p$ doit \^etre positif ou nul. Supposons maintenant que $p$ soit nul. Alors $\widetilde f_{I,z'}^q(\widetilde z)\in \widetilde U_D$, et donc $\widetilde f_{I,z'}^{q-1}(\widetilde z)\in \widetilde f_{I,z'}^{-1}(\widetilde U_D)$. Ceci est absurde car le disque $\widetilde f_{I,z'}^{-1}(\widetilde U_D)$ est \`a gauche au sens strict de $\widetilde\lambda_0$ par construction de $U$, et  $\widetilde f_{I,z'}^{q-1}(\widetilde z)$ est \`a droite au sens strict de $\widetilde\lambda_0$. Par cons\'equent, l'entier $p=\alpha_{U,z'}(z)$ doit \^etre strictement positif.

\medskip

Nous avons ainsi montr\'e que l'entier $\alpha_{U,z'}(z)$ est strictement positif pour tout point $z\in \mbox{Rec}^+(f)\cap U$. Par ailleurs,  le point $z_0$ est dans le support de $\mu$, donc $\mu$ charge le disque $U$. En utilisant le lemme~\ref{c.nombre-rotation-moyen-tend-vers-zero}, on en d\'eduit $\mu(\{z\in \mathbb{R}^2\;,\; \rho_{z'}(z)>0\})>0$,
ce qui montre l'assertion~(5) dans le cas consid\'er\'e.

\bigskip

Le cas o\`u $\infty\in\widehat\alpha(\lambda_0)$ et $z'\in\widehat\omega(\lambda_0)$ se traite \'evidemment de mani\`ere similaire. Dans ce cas, on trouve un disque libre $V$ tel que $\mu(V)>0$ et $\alpha_{V,z'}(z)<0$ pour tout point $z\in V\cap \mathrm{Rec}^+(f)$. On en d\'eduit que $\mu(\{z\in  \mathbb{R}^2\;,\; \rho_{z'}(z)<0\})>0$.

\bigskip

Dans le cas o\`u $\infty\in\widehat\omega(\lambda_0)$ et $z'\not\in\widehat\alpha(\lambda_0)$, on choisit un point fixe $z''\in\widehat\alpha(\lambda_0)\cap X$, on reprend la preuve du premier cas, en y rempla\c cant $z'$ par $z''$, et on obtient que \hbox{$\mu(\{z\in  \mathbb{R}^2\;,\; \rho_{z''}(z)>0\})>0$}. Le seul point \`a modifier est la preuve de la seconde affirmation du lemme~\ref{l.va vers-la-droite}. Notons que les arguments utilis\'es dans le premier cas ne s'appliquent pas ici, car l'isotopie $I$ ne fixe pas le point $z''$, et les chemins $\gamma_{I,z}$ et $\gamma_{I',z}$ ne sont pas homotopes dans l'anneau $A_{z''}$. On raisonne comme suit. Soit $z$ un point de $A\setminus X$, et $\widetilde z$ un relev\'e de $z$ dans $\widetilde A_{z''}$. On sait que le chemin $\gamma_{I',z}$ ne passe pas par le point $z''$ (puisque $z''\in X$). On consid\`ere alors le rel\`evement $\widetilde{f'}$ de $f\vert_{A_{z''}}$ \`a $\widetilde A_{z''}$ tel que le chemin $\gamma_{I',z}:t\mapsto f_t'(z)$ se rel\`eve dans  $\widetilde A_{z''}$ en un chemin joignant $\widetilde z$ \`a $\widetilde{f'}(\widetilde z)$. D'apr\`es l'assertion~(5) du th\'eor\`eme~\ref{theo.Brouwer-feuillete-equivariant}, il existe une isotopie $I''=(f_t'')_{t\in [0,1]}$ joignant l'identit\'e \`a $f$ dans $\mathrm{Homeo}(\mathbb{R}^2)$, qui fixe les  points $z'$ et $z''$, et telle que les chemins $\gamma_{I',z}$ et $\gamma_{I'',z}$ sont homotopes dans $\mathbb{R}^2\setminus\{z',z''\}$. Il en r\'esulte que le rel\`evement $\widetilde{f'}$ co\"{\i}ncide avec le rel\`evement $\widetilde f_{I'',z''}$. 
D'autre part, l'assertion~(1) du th\'eor\`eme~\ref{theo.Brouwer-feuillete-equivariant} nous dit que le lacet $\gamma_{I,z''}$ est homotope \`a un point dans $A_{z'}$. Ceci implique que $I$ est homotope \`a une isotopie qui fixe les points $\infty$, $z'$ et $z''$, et donc est homotope \`a $I''$. Par suite, les rel\`evements $\widetilde f_{I,z''}$ et $\widetilde f_{I'',z''}$ co\"{\i}ncident. Enfin, l'assertion~(4) du th\'eor\'eme~\ref{theo.Brouwer-feuillete-equivariant} affirme que le chemin $\gamma_{I',z}$ est homotope dans $A_{z'}\setminus X$ \`a un chemin $\gamma_z$ positivement transverse au feuilletage $\mathcal{F}$.  Puisque le chemin $\gamma_z$ est homotope \`a $\gamma_{I',z}$ dans $A_{z'}\setminus X$, donc dans $A_{z''}$, il se rel\`eve dans $\widetilde A_{z''}$ en un chemin $\widetilde\gamma_z$ joignant $\widetilde z$ \`a $\widetilde{f'}(\widetilde z)=\widetilde f_{I'',z''}(\widetilde z)=\widetilde f_{I,z''}(\widetilde z)$.  Ainsi nous avons montr\'e l'existence d'un chemin positivement transverse au feuilletage $\widetilde{\mathcal{F}}$ joignant  $\widetilde z$ \`a $\widetilde f_{I,z''}(\widetilde z)$ dans $\widetilde A_{z''}$. Ce fait \'etant \'etabli, on peut terminer la preuve du lemme~\ref{l.va vers-la-droite} et la preuve de l'assertion~(5) exactement comme dans le premier cas.

\bigskip

Le cas  o\`u $\infty\in\widehat\alpha(\lambda_0)$ et $z'\not\in\widehat\alpha(\lambda_0)$ se traite \'evidemment de mani\`ere similaire~: on choisit un point $z''\in\widehat\omega(\lambda_0)\cap X$ et on montre que $\mu(\{z\in  \mathbb{R}^2\;,\; \rho_{z''}(z)<0\})>0$.

\bigskip

Il reste \`a \'etudier le cas o\`u $\infty\in A=\mathbb{S}^2\setminus (\widehat \alpha(\lambda_0)\sqcup\widehat \omega(\lambda_0))$. Dans ce cas, on fixe $z''\in \widehat\alpha(\lambda_0)\cap X$ et $z'''\in\widehat\omega(\lambda_0)\cap X$. On rappelle que la notation $A_{z'',z'''}$ d\'esigne l'anneau $\mathbb{S}^2\setminus\{z'',z'''\}$, et que $\widetilde f_{z'',z'''}$ d\'esigne le rel\`evement de $f\vert_{A_{z'',z'''}}$ \`a $\widetilde A_{z'',z'''}$ qui fixe les relev\'es de $\infty$. On reprend la preuve du premier cas, en rempla\c cant l'anneau $A_{z'}$ par l'anneau $A_{z'',z'''}$, et le rel\`evement $\widetilde f_{I,z'}$ par le rel\`evement $\widetilde f_{z'',z'''}$. \`A nouveau, le seul point \`a modifier est la preuve de la seconde affirmation du lemme~\ref{l.va vers-la-droite}. On raisonne comme suit.  Soit $z$ un point de $A\setminus X$, et $\widetilde z$ un relev\'e de $z$ dans $\widetilde A_{z'',z'''}$. On consid\`ere le rel\`evement $\widetilde{f'}$ de $f\vert_{A_{z'',z'''}}$ \`a $\widetilde A_{z'',z'''}$ tel que le chemin $\gamma_{I',z}$ se rel\`eve en un chemin de $\widetilde z$ \`a $\widetilde{f'}(\widetilde z)$. En utilisant l'assertion~(5) du th\'eor\`eme~\ref{theo.Brouwer-feuillete-equivariant} comme dans le cas o\`u $\infty\in\widehat\omega(\lambda_0)$ et $z'\not\in\widehat\alpha(\lambda_0)$, on voit que $\widetilde{f'}$ co\"{\i}ncide avec le rel\`evement canonique $\widetilde f_{z'',z'''}$ (rappelons que ce dernier est le rel\`evement de $f\vert_{A_{z'',z'''}}$ qui fixe les relev\'es du point $\infty$).
L'assertion~(4) du th\'eor\'eme~\ref{theo.Brouwer-feuillete-equivariant} assure alors l'existence d'un chemin $\widetilde\gamma_z$ joignant $\widetilde z$ \`a $\widetilde{f'}(\widetilde z)=\widetilde f_{z'',z'''}(z)$ dans $\widetilde A_{z'',z'''}$, positivement transverse au relev\'e $\widetilde{\mathcal{F}}$ du feuilletage $\mathcal{F}$. On peut d\`es lors reprendre la fin de la preuve du premier cas (en rempla\c cant l'anneau $A_{z'}$ par l'anneau $A_{z'',z'''}$, et le rel\`evement $\widetilde f_{I,z'}$ par le rel\`evement $\widetilde f_{z'',z'''}$).  En utilisant le fait~\ref{f.nombre-rotation-8}, on en d\'eduit que $\mu(\{z\in  \mathbb{R}^2\;,\; \rho_{z''}(z)-\rho_{z'''}(z)>0\})>0$. On a donc $\mu(\{z\in  \mathbb{R}^2\;,\; \rho_{z''}(z)> 0\}>0$ ou $\mu(\{z\in  \mathbb{R}^2\;,\; \rho_{z'''}(z)<0\})>0$. Ceci termine la preuve de l'assertion~(5), et donc aussi de la proposition~\ref{theo.main-2}.
\end{proof}

\begin{remark}
Soit $\DD^{2n}$ le disque unit\'e de $\RR^{2n}$. C. Viterbo a montr\'e que tout diff\'eomorphisme hamiltonien de $\DD^{2n}$, qui vaut l'identit\'e au bord, admet un point fixe d'action symplectique non-nulle (\cite[proposition 4.2]{Viterbo1992}). Pour $n=1$, les diff\'eomorphismes hamiltoniens du disque sont simplement ceux qui pr\'eservent l'aire euclidienne usuelle et l'action symplectique n'est autre que l'int\'egrale du nombre de rotation. Le r\'esultat de Viterbo se traduit donc de la mani\`ere suivante. Pour tout diff\'eomorphisme $f$ du disque $\DD^2$ qui vaut l'identit\'e au bord et qui pr\'eserve l'aire euclidienne usuelle, si $I$ est une isotopie joignant l'identit\'e \`a $f$, il existe un point $z'\in\mathrm{Fix}(f)$ tel que l'int\'egrale de la fonction nombre de rotation $z\mapsto \rho_{I,z'}(z)$ par rapport \`a l'aire euclidienne est non-nulle. Ainsi le r\'esultat de Viterbo fournit donc une preuve de l'assertion~(5) de notre proposition~\ref{theo.main-2}, dans le cas o\`u l'hom\'eomorphisme $f$ se prolonge en un $C^1$-diff\'eomorphisme de la sph\`ere $\mathbb{S}^2$ et o\`u la mesure $\mu$ est donn\'ee par une forme volume. On se ram\`ene au disque $\DD^2$ en \'eclatant le point $\infty$.  
\end{remark}

\section{Les propri\'et\'es (P1), (P2) et les $C^1$-diff\'eomorphismes de la sph\`ere} 
\label{s.P1-P2-diffeos}

Le but de cette partie est de montrer, comme annonc\'e dans la proposition~\ref{p.extension-C1}, que les propri\'et\'es (P1) et (P2) sont v\'erifi\'ees par tout hom\'eomorphisme du plan qui s'\'etend en un $C^1$-diff\'eomorphisme de la sph\`ere.

\begin{proof}[Preuve de la proposition~\ref{p.extension-C1}]
Supposons donc que $f$ admet une extension de classe $C^1$ \`a $\mathbb{S}^2$ et commen\c cons par montrer que $f$ v\'erifie (P2). Nous n'avons besoin que de la diff\'erentiabilit\'e de $f$ en $\infty$ et plus pr\'ecis\'ement de l'existence d'une compactification $\overline{\RR^2}$ de $\RR^2$ par ajout d'un cercle $\Sigma$ \`a l'infini tel que $f$ se prolonge en un hom\'eomorphisme de $\overline{\RR^2}$.  Fixons $z'\in  \mathrm{Fix}(f)$. L'hom\'eomorphisme $f\vert_{A_{z'}}$ se prolonge en un hom\'eomorphisme de $A_{z'}\sqcup\Sigma$. Le rev\^etement universel de $A_{z'}\sqcup\Sigma$ s'\'ecrit $\widetilde  A_{z'}\sqcup\widetilde{\Sigma}$, o\`u $\widetilde\Sigma$ est le rev\^etement universel de $\Sigma$ et $\widetilde f_{I,z'}$ se prolonge en un hom\'eomorphisme $\overline{\widetilde f_{I,z'}}$ de $\widetilde  A_{z'}\sqcup\widetilde{\Sigma}$. Si $\mathrm{Fix}(f)$ est compact, la propri\'et\'e (P2) est \'evidemment v\'erifi\'ee. Sinon il y a des points fixes de l'extension de $f\vert_{A_{z'}}$ sur ${\Sigma}$ et il existe donc un entier $p$ tel que $\overline{\widetilde f_{I,z'}}(\widetilde z)=\overline{T_{z'}}^p(\widetilde z)$ pour tout relev\'e $\widetilde z$ d'un point fixe appartenant \`a $\Sigma$. On a donc \'egalement  $\widetilde f_{I,z'}(\widetilde z)=T_{z'}^p(\widetilde z)$ pour tout relev\'e $\widetilde z$ d'un point fixe proche de $\Sigma$, ce qui n'est rien d'autre que la propri\'et\'e (P2).

\medskip

Montrons maintenant que (P1) est v\'erifi\'ee. La fonction $\mathrm{Enlace}_I$, restreinte \`a l'ensemble des couples de points fixes distincts, \'etant localement constante, il suffit de d\'emontrer les deux assertions suivantes~:

\smallskip

\noindent\textbf{(i)}\enskip Il existe un voisinage $W$ de $\infty$ et une constante $M$ telle que $\vert\mathrm{Enlace}_I(z,z')\vert\leq M$ pour tout $z\in\mathrm{Fix}(f)$ et tout $z'\in\mathrm{Fix}(f)\cap W$ distincts.

\smallskip

\noindent\textbf{(ii)}\enskip Pour tout $z''\in \mathrm{Fix}(f)$ il existe un voisinage point\'e $W$ de $z''$ et une constante $M$ telle que $\vert\mathrm{Enlace}_I(z,z')\vert\leq M$ pour tous $z,z'\in\mathrm{Fix}(f)\cap W$ distincts.

\smallskip

Remarquons tout d'abord qu'en vertu de la formule~\eqref{e.change-isotopie-2} dans l'introduction,  les assertions~(i) et~(ii) ne d\'ependent pas du choix de l'isotopie~$I$. On pourra donc supposer que $I$ est une isotopie adapt\'ee. La propri\'et\'e~(i) est alors un cas particulier de l'assertion~(4) de la proposition~\ref{theo.main-2}~: on rappelle en effet que, si $z$ et $z'$ sont deux points fixes de $f$, alors $\rho_{z'}(z)=\mathrm{Enlace}_I(z,z')$.

Montrons maintenant (ii) en donnant \`a $z''$ le r\^ole pris auparavant par $\infty$. On peut toujours supposer que l'isotopie $I$ fixe $z''$. Puisque $f$ est diff\'erentiable en $z''$, on peut compactifier l'anneau $A_{z''}$ par ajout d'un cercle au bout correspondant \`a $z''$ de telle fa\c con que $f_{z''}$ se prolonge en un hom\'eomorphisme. Les arguments utilis\'es dans la v\'erification de la propri\'et\'e (P1) nous permettent de supposer, quitte \`a changer d'isotopie, que pour tout point fixe $z\not=z''$ proche de $z''$, on a $\mathrm{Enlace}_I{(z,z'')}=0$.  Le preuve de~(i) nous permet quant \`a elle de dire que, si $z'\not=z''$ est proche de $z''$, alors, pour tout point fixe $z\not\in\{z',z''\}$, on a $\left|\rho_{z'}(z)-\rho_{z''}(z)\right|=\left|\mathrm{Enlace}_I(z,z')-\mathrm{Enlace}_I(z,z'')\right|\leq M$. On en d\'eduit que $\vert\mathrm{Enlace}_I(z,z')\vert \leq M$ si $z$ et $z'$ sont tous deux proches de $z''$. 
\end{proof}

\begin{remark}
La fonction $(z,z')\mapsto\mathrm{Enlace}_I(z,z')$ est l'un des objets utilis\'es par J.-M. Gambaudo et  \'E. Ghys dans~\cite{GambaudoGhys2004} pour construire des quasi-morphismes sur le groupe des diff\'eomorphismes d'une surface pr\'eservant l'aire. On notera que les arguments de la preuve du lemme 4.1 de~\cite{GambaudoGhys2004} fournissent une preuve (diff\'erente de celle propos\'ee ci-dessus) du fait que tout hom\'eomorphisme du plan qui se prolonge en  un $C^1$-diff\'eomorphisme de $\mathbb{S}^2$ v\'erifie la propri\'et\'e~(P1).
\end{remark}

\section*{Appendice. Quelques exemples d'hom\'eomorphismes du plan qui v\'erifient, ou pas, les propri\'et\'es (P1) et (P2)}

Le but de cet appendice est de tenter d'\'eclairer le sens des propri\'et\'es (P1) et (P2), en d\'ecrivant quelques exemples d'hom\'eomorphismes du plan qui satisfont, ou ne satisfont pas, l'une et/ou l'autre de ces propri\'et\'es. Nous identifions $\mathbb{R}^2$ \`a $\mathbb{C}$, afin de pouvoir utiliser les notations complexes. Dans les exemples suivants, nous notons $z_n$ le point d'abscisse $n$ sur l'axe r\'eel, et $B_n$ la boule euclidenne de centre $z_n$ de rayon $\frac{1}{4}$.

\subsection*{Exemple~1. Un hom\'eomorphisme  qui ne satisfait ni (P1), ni (P2).}
 Consid\'erons le diff\'eomorphisme $f\in \mathrm{Diff}^\infty_+(\mathbb{R}^2)$ d\'efini (dans une coordonn\'ee complexe) par la formule
$$f\left(re^{2i\pi\theta}\right)=re^{2i\pi(\theta+r)}.$$
 On construit facilement une isotopie $I=(f_t)_{t\in [0,1]}$ joignant l'identit\'e \`a $f$~: le diff\'eomorphisme~$f_t$ est obtenu en rempla\c cant $\theta+r$ par $\theta+t\,r$ dans la formule ci-dessus. Pour tout $n\in\ZZ$, le point $z_n$ est fixe par $f$, et on a $\mathrm{Enlace}_I(0,z_n)=\mathrm{Tourne}_I(z_n)=n$. Par cons\'equent, $f$ ne satisfait ni (P1) ni~(P2).

\subsection*{Exemple~2. Un hom\'eomorphisme qui satisfait (P1), mais pas (P2).}
 Consid\'erons maintenant le diff\'eomorphisme $f\in \mathrm{Diff}^\infty_+(\mathbb{R}^2)$ d\'efini par la formule 
$$f\left(re^{2i\pi\theta}\right)=re^{2i\pi\left(\theta+\sin\left(\pi r+ \frac{\pi}{2}\right)\right)}.$$
On construit une isotopie $I=(f_t)_{t\in [0,1]}$ joignant l'identit\'e \`a $f$, comme dans l'exemple pr\'ec\'edent~: le diff\'eomorphisme $f_t$ est obtenu en rempla\c cant $\sin\left(\pi r+ \frac{\pi}{2}\right)$ par $t\,\sin\left(\pi r+ \frac{\pi}{2}\right)$ dans la formule d\'efinissant $f$. Si $z$ et $z'$ sont deux points distincts du plan, on v\'erifie facilement que  
$$\mathrm{Enlace}_I(z,z')=\sin\left(\pi r+ \frac{\pi}{2}\right) $$
o\`u $r=\max(|z|,|z'|)$. En particulier, $f$ v\'erifie (P1). Rappelons que la notation $z_n$ d\'esigne le point d'abscisse $n$ sur l'axe r\'eel. Pour montrer que $f$ ne v\'erifie pas (P2), il suffit de remarquer que $z_n$ est fixe par $f$ pour tout $n\in\ZZ$, et que 
$$\mathrm{Tourne}_I(z_n)=\sin\left(\pi n+ \frac{\pi}{2}\right) = \left\{ 
\begin{array}{rl}
1 & \mbox{si $n$ est pair} \\
-1 & \mbox{si $n$ est impair}
\end{array}\right.$$

\subsection*{Exemple~3. Un hom\'eomorphisme qui satisfait (P2), mais pas (P1).}
Pour chaque entier $n$, consid\'erons une  fonction $\alpha_n:[0,1/4]\to\mathbb{R}$, lisse, qui s'annule au voisinage de $0$ et $1/4$, et qui vaut $n$ en $1/8$. Rappelons que $z_n$ d\'esigne le point d'abscisse $n$ sur l'axe r\'eel, et $B_n$ la boule euclidienne ferm\'ee de centre $z_n$ et de rayon $1/4$.  Consid\'erons alors le diff\'eomorphisme $f\in\mathrm{Diff}^\infty_+(\mathbb{R}^2)$ qui est d\'efini sur la boule~$B_n$ par la formule
$$f\left(z_n+re^{2i\pi\theta}\right)=z_n+re^{2i\pi(\theta+\alpha_n(r))},$$
et qui co\"{\i}ncide avec l'identit\'e sur $\mathbb{R}^2\setminus\bigsqcup_{n\in\ZZ} B_n$. On construit une isotopie $I=(f_t)_{t\in [0,1]}$ joignant l'identit\'e \`a $f$ comme dans les exemples pr\'ec\'edents~: le diff\'eomorphisme $f_t$ est obtenu en rempla\c cant $\alpha_n(r)$ par $t\,\alpha_n(r)$ dans la formule d\'efinissant $f$. Quel que soit $n\in\ZZ$, si on note $z_n':=z_n+\frac{1}{8}$, alors $z_n'$ est un point fixe de $f$, et on a 
$$\mathrm{Enlace}_I(z_n,z_n')=n.$$
En particulier, $f$ ne v\'erifie pas (P1). Par ailleurs, $f_t$ co\"{\i}ncide avec l'identit\'e sur $\mathbb{R}^2\setminus\bigsqcup_{n\in\ZZ} B_n$, et tout lacet contenu dans $\bigsqcup_{n\in\ZZ} B_n$ qui fait le tour de l'origine est en fait contenu dans la boule $B_0$. Ceci montre que $\mathrm{Tourne}_I(z)=0$ pour tout point $z\in \mathrm{Fix}(f)$ situ\'e hors de $B_0$. En particulier, $f$ v\'erifie~(P2).

\subsection*{Exemple~4. Un hom\'eomorphisme qui satisfait (P1) et (P2), mais ne s'\'etend pas en un $C^1$-diff\'eomorphisme de la sph\`ere.}
Fixons $\theta_0\in\RR\setminus\ZZ$, et consid\'erons une fonction lisse $\alpha:[0,1/4]\to\mathbb{R}$ qui vaut $\theta_0$ au voisinage de $0$, et qui s'annule au voisinage de $1/4$. Consid\'erons alors le diff\'eomorphisme $f\in\mathrm{Diff}^\infty_+(\mathbb{R}^2)$, d\'efini par la formule
$$f\left(z_n+re^{2i\pi\theta}\right)=z_n+re^{2i\pi(\theta+\alpha(r))}$$
sur la boule $B_n$, et qui co\"{\i}ncide avec l'identit\'e sur $\mathbb{R}^2\setminus\bigsqcup_{n\in\ZZ} B_n$. On construit une isotopie $I=(f_t)_{t\in [0,1]}$ joignant l'identit\'e \`a $f$ en rempla\c cant $\alpha(r)$ par $t\,\alpha(r)$  dans la formule d\'efinisssant $f$. Le m\^eme argument que dans l'exemple~3 montre que $f$ v\'erifie (P2). Si $z,z'$ sont deux points distincts du plan, on a 
$$\mathrm{Enlace}_I(z,z')=\left\{ 
\begin{array}{ll}
\alpha\left(\max(|z-z_n|,[z'-z_n|)\right) & \mbox{ si }~z\in B_n~\mbox{ et }z'\in B_n  \\
0 & \mbox{ sinon}
\end{array}\right.$$
Par cons\'equent, $f$ v\'erifie (P1).
Notons maintenant $\bar f$ l'extension de $f$ \`a $\mathbb{S}^2=\mathbb{R}^2\cup\{\infty\}$ obtenue en posant $f(\infty)=\infty$. Alors $\bar f$ est un diff\'eomorphisme de $\mathbb{S}^2$ pour la structure diff\'erentielle usuelle, mais la diff\'erentielle de $\bar f$ n'est pas continue en $\infty$~: en effet, la diff\'erentielle de $\bar f$ en $\infty$ est l'identit\'e, mais la diff\'erentielle en $z_n$ est la rotation d'angle $2\pi\theta_0$ pour tout $n$.

\begin{remark}
Le diff\'eomorphisme $f$ construit ci-dessus commute avec la translation $z\mapsto z+1$. Ainsi, aucune partie compacte non-vide de $\RR^2$ n'est invariante par tout \'el\'ement de $\mathrm{Homeo}(\RR^2)$ qui commute avec $f$. 
Notons \'egalement que la restriction de la mesure de Lebesgue \`a la boule $B_n$ est une mesure de masse finie $f$-invariante, pour tout $n\in\ZZ$.  Ceci montre que l'hypoth\`ese ``$f$ s'\'etend en un $C^1$-diff\'eomorphisme de $\mathbb{S}^2$" est n\'ecessaire dans l'assertion~2 du th\'eor\`eme~\ref{theo.main}, comme nous l'avons annonc\'e dans la remarque~\ref{r.contre-exemple-pas-C1}. 
\end{remark}

\subsection*{Exemple~5. Un hom\'eomorphisme qui satisfait (P1) et (P2), pour lequel il existe un point $z'\in\mathrm{Fix}(f)$ tel que la fonction $z\mapsto \rho_{z'}(z)$ n'est pas born\'ee.} 
Nous avons montr\'e dans la remarque~\ref{r.nombre-rotation-pas-borne} que, pour tout hom\'eomorphisme $f$ du plan pr\'eservant l'orientation et v\'erifiant les propri\'et\'es (P1) et (P2), et tout point $z'\in\mathrm{Fix}(f)$, la fonction $z\mapsto \rho_{z'}(z)$ est born\'ee sur tout disque libre pour $f$. Cette fonction n'est cependant pas globalement born\'ee en g\'en\'eral comme le montre l'exemple tr\`es simple suivant.  Consid\'erons un $C^\infty$-diff\'eomorphisme $\rho$ de la demi-droite $[0,+\infty[$ qui fixe chaque entier, et tel que, pour tout $n\in\NN$, et tout $r\in ]n,n+1[$,  on a $\rho^k(r)\rightarrow n$ quand $k\to -\infty$ et $\rho^k(r)\rightarrow n+1$ quand $k\to +\infty$. Consid\'erons alors le diff\'eomorphisme  $f\in \mathrm{Diffeo}^\infty_+(\mathbb{R}^2)$ d\'efini par la formule
$$f\left(re^{2i\pi\theta}\right)=\rho(r)e^{2i\pi(\theta+r+\frac{1}{2})}.$$
 On construit facilement une isotopie $I=(f_t)_{t\in [0,1]}$ joignant l'identit\'e \`a $f$~: le diff\'eomorphisme~$f_t$ est obtenu en rempla\c cant $\theta+r$ par $\theta+t\,r$, et $\rho(r)$ par $(1-t)r+t\rho(r)$ dans la formule ci-dessus. Notons $C_n$ le cercle centr\'e \`a l'origine de rayon $n$. Pour tout point $z$ tel que $n<|z|<n+1$, l'$\alpha$-limite de l'orbite pour $f$ du point $z$ est le cercle $C_n$, et l'$\omega$-limite de cette orbite est le cercle $C_{n+1}$. Par ailleurs, la restriction de $f$ au cercle $C_n$ est une rotation d'angle $2\pi(n+1/2)$. En particulier, $f$ n'a aucun point fixe sur $C_n$. Ceci montre que l'origine est le seul point fixe de $f$, ce qui prouve que $f$ v\'erifie les propri\'et\'es~(P1) et~(P2). Ceci montre \'egalement que, pour tout point $z$ tel que $|z|=n$, on a $\rho_{0}(z)=n+1/2$. En particulier, la fonction $z\mapsto \rho_{0}(z)$ n'est pas born\'ee. Remarquons qu'on peut \'egalement exhiber une mesure $\mu$, de masse finie, $f$-invariante, telle que la fonction $z\mapsto \rho_{0}(z)$ n'est pas born\'ee sur le support de $\mu$~: il suffit de prendre par exemple $\mu:=\sum_{n>0} 2^{-n} \mu_n$ o\`u $\mu_n$ est la mesure de Lebesgue (normalis\'ee pour avoir une masse totale \'egale \`a $1$) sur le cercle $C_n$.
 
 \subsection*{Exemple~5bis. Un hom\'eomorphisme qui satisfait (P1) et (P2), pour lequel il existe $z'\in\mathrm{Fix}(f)$ tel que la fonction $z\mapsto \rho_{z'}(z)$ n'est pas born\'ee sur les compacts de $\RR^2$.} 
 Dans l'exemple~5 ci-dessus, la fonction $z\mapsto \rho_{0}(z)$ est born\'ee sur tout compact de $\RR^2$. Il est cependant facile de modifier cet exemple, pour que ce ne soit plus le cas. Il suffit de prendre maintenant un hom\'eomorphisme $\rho$ de la demi-droite $[0,+\infty[$ qui fixe le point $1/n$ pour tout $n>0$, et tel que, pour tout $r\in ]1/(n+1),1/n[$,  on a $\rho^k(r)\rightarrow 1/n$ quand $k\to -\infty$ et $\rho^k(r)\rightarrow 1/(n+1)$ quand $k\to +\infty$, puis de consid\'erer l'hom\'eomorphisme $f$ d\'efini par la m\^eme formule que dans l'exemple~5. L'hom\'eomorphisme ainsi obtenu n'est pas diff\'erentiable \`a l'origine, mais v\'erifie les propri\'et\'es (P1) et (P2) pour la m\^eme raison que dans l'exemple~5~: l'origine est le seul point fixe de $f$. D'autre part, pour tout $z$ tel que $|z|=1/n$, on a $\rho_{0}(z)=n+1/2$, ce qui implique que la fonction $z\mapsto \rho_{0}(z)$ n'est pas born\'ee au voisinage de l'origine.

\subsection*{Exemple~6. Un hom\'eomorphisme qui satisfait (P1) et (P2), tel que la fonction $z\mapsto \rho_{z'}(z)$ est born\'ee pour tout $z'\in\mathrm{Fix}(f)$, mais tel que la fonction $(z,z')\mapsto \rho_{z'}(z)$ n'est pas born\'ee.} 
Nous avons montr\'e au cours de la preuve de l'assertion~(3) de la proposition~\ref{theo.main-2} que, pour tout hom\'eomorphisme qui v\'erifie les propri\'et\'es~(P1) et~(P2), la fonction $(z,z')\mapsto \rho_{z'}(z)$ est localement born\'ee sur $(\RR^2\setminus\mathrm{Fix}(f))\times \mathrm{Fix}(f)$. Les deux exemples pr\'ec\'edents montrent, qu'\`a $z'$ fix\'e, la fonction $z\mapsto \rho_{z'}(z)$ n'est pas globalement born\'ee en g\'en\'eral. Nous allons maintenant d\'ecrire un exemple tel que la fonction $z\mapsto \rho_{z'}(z)$ est born\'ee sur $\RR^2\setminus\{z'\}$ pour tout $z'\in\mathrm{Fix}(f)$, mais tel que la fonction $(z,z')\mapsto \rho_{z'}(z)$ n'est pas globalement born\'ee sur $(\RR^2\setminus\mathrm{Fix}(f))\times \mathrm{Fix}(f)$.

Comme dans l'exemple~4, $z_n$ d\'esigne le point d'abscisse $n$ sur l'axe r\'eel, et $B_n$ la boule euclidienne ferm\'ee de centre $z_n$ et de rayon $1/4$, pour $n\in\ZZ$. Consid\'erons un hom\'eomorphisme $\rho$ de l'intervalle $[0,1/4]$, qui fixe les points $0$, $1/8$ et $1/4$, et tel que l'orbite future pour $\rho$ de tout point de $]0,1/4[$ tend vers le point $1/8$.  Pour tout $n$, consid\'erons une fonction lisse $\alpha_n:[0,1/4]\to\mathbb{R}$ qui s'annule au voisinage de $0$ et au voisinage de $1/4$, et telle que $\alpha_n(1/8)= n+1/2$. Consid\'erons alors le diff\'eomorphisme $f\in\mathrm{Diff}^\infty_+(\mathbb{R}^2)$, d\'efini par la formule
$$f\left(z_n+re^{2i\pi\theta}\right)=z_n+\rho(r)e^{2i\pi(\theta+\alpha_n(r))}$$
sur la boule $B_n$, et qui co\"{\i}ncide avec l'identit\'e sur $\mathbb{R}^2\setminus\bigsqcup_{n\in\ZZ} B_n$. On construit une isotopie $I=(f_t)_{t\in [0,1]}$ joignant l'identit\'e \`a $f$ en rempla\c cant $\alpha_n(r)$ par $t\,\alpha_n(r)$, et $\rho(r)$ par $(1-t)r+t\rho(r)$, dans la formule d\'efinissant $f$. Le choix de l'hom\'eomorphisme $\rho$ implique que, quel que soit $n\in\ZZ$, l'orbite future pour $f$ de tout point $z\in B_n\setminus\{z_n\}$ va s'accumuler sur le cercle de centre $z_n$ et de rayon $1/8$. Sur ce cercle, $f$ agit comme une rotation d'angle $2\pi(n+1/2)$. En particulier, $f$ n'a pas de point fixe sur ce cercle. Ainsi, les seuls points fixes de $f$ sont les points $z_n$, et les points de $\RR^2\setminus \bigsqcup_{n\in\ZZ}B_n$. On en d\'eduit facilement que, pour tout couple de points distincts $z,z'\in\mathrm{Fix}(f)$, on a $\mathrm{Enlace}_I(z,z')=0$. En particulier, $f$ v\'erifie la propri\'et\'e~(P1). Les m\^emes arguments que dans l'exemple~4 montrent que $f$ v\'erifie \'egalement la propri\'et\'e~(P2). Par contre, la fonction $(z,z')\mapsto \rho_{z'}(z)$ n'est pas born\'ee~: en effet, si $w_n$ est un point du cercle de centre $z_n$ et de rayon $1/8$, alors on a clairement $\rho_{z_n}(w_n)=n+1/2$.


\begin{thebibliography}{99}
\bibitem{Bonatti1989}
Bonatti, Christian. Un point fixe commun pour des diff\'eomorphismes commutants de $S^2$.
\textit{Ann. of Math. (2)} \textbf{129} (1989), no. 1, 61--69. 

\bibitem{Bonatti1990}
Bonatti, Christian. Diff\'eomorphismes commutants des surfaces et stabilit\'e des fibrations en tores. 
\textit{Topology} \textbf{29} (1990), no. 1, 101--126. 

\bibitem{BrownKister1984}
Brown, Morton  and Kister, James M. Invariance of complementary domains of a fixed point set.
\textit{Proc. Amer. Math. Soc.} \textbf{91} (1984), no.	3, 503--504.

\bibitem{Brouwer1912}
Brouwer, Luitzen E. J.
Beweis des ebenen Translationssatzes.
\textit{Math. Ann.} \textbf{72} (1912), 37--54.

\bibitem{DruckFangFirmo2002}
Druck, Suely; Fang, Fuquan; Firmo, Sebasti\~ao. Fixed points of discrete nilpotent group actions on $S^2$. 
\textit{Ann. Inst. Fourier} \textbf{52} (2002), no. 4, 1075--1091. 

\bibitem{Firmo2005}
Firmo, Sebasti\~ao. A note on commuting diffeomorphisms on surfaces. \textit{Nonlinearity} \textbf{18} (2005), no. 4, 
1511--1526. 

\bibitem{Franks1988}
Franks, John. Generalizations of the Poincar\'e-Birkhoff theorem. \textit{Ann. of Math. (2)} \textbf{128} (1988), no. 1, 139--151. 

\bibitem{FranksHandelParwani2007}
Franks, John~; Handel, Michael~; Parwani, Kamlesh. Fixed points of abelian actions on $S^2$. \textit{Ergodic Theory Dynam. Systems} \textbf{27} (2007), no. 5, 1557--1581. 

\bibitem{FranksHandelParwani2007-2}
 Franks, John; Handel, Michael; Parwani, Kamlesh.
 Fixed points of abelian actions. \textit{J. Mod. Dyn.} \textbf{1} (2007), no. 3, 443--464.
 
 \bibitem{GambaudoGhys2004}
  Gambaudo, Jean-Marc~; Ghys, \'Etienne. Commutators and diffeomorphisms of surfaces. 
  \textit{Ergodic Theory Dynam. Systems} \textbf{24} (2004), no. 5, 1591--1617. 
 
 \bibitem{Guillou1994}
 Guillou, Lucien. 
 Th\'eor\`eme de translation plane de Brouwer et g\'en\'eralisations du th\'eor\`eme de Poincar\'e-Birkhoff. 
 \textit{Topology} \textbf{33} (1994), no. 2, 331--351.

\bibitem{Handel1992}
Handel, Michael.
Commuting homeomorphisms of $S^2$. \textit{Topology} \textbf{31} (1992), no. 2, 293--303. 

\bibitem{Jaulent2010}
Jaulent, Olivier.
\textit{Existence d'un feuilletage positivement transverse \`a un hom\'eomorphisme de surface.} En pr\'eparation.

\bibitem{Kneser1926}
Kneser, Hellmuth.
Die Deformationss\"atze der einfach zusammenhŠngenden Fl\"achen. 
\textit{Math. Z.} \textbf{25} (1926), no. 1, 362--372. 

\bibitem{LeCalvez2005}
Le Calvez, Patrice. 
Une version feuillet\'ee \'equivariante du th\'eor\`eme de translation de Brouwer. \textit{Publ. Math. Inst. Hautes \'etudes Sci.} \textbf{102} (2005), 1--98.

\bibitem{LeRoux2001}
Le Roux, Fr\'ed\'eric.
\'Etude topologique de l'espace des hom\'eomorphismes de Brouwer. I. 
\textit{Topology} \textbf{40} (2001), no. 5, 1051--1087.

\bibitem{Lima1964}
Lima, Elon L. 
Commuting vector fields on $S^{2}$. \textit{Proc. Amer. Math. Soc.} \textbf{15} (1964), 138--141.

\bibitem{Mann2011}
Mann, Kathryn.
\textit{Bounded orbits and global fixed points for groups acting on the plane}.
\texttt{arXiv:1103.5060}.

\bibitem{Plante1986}
Plante, Joseph F. Fixed points of Lie group actions on surfaces. \textit{Ergodic Theory Dynam. Systems} \textbf{6} (1986), no. 1, 149--161. 

\bibitem{Viterbo1992}
Viterbo, Claude.
Symplectic topology as the geometry of generating functions. 
\textit{Math. Ann.} \textbf{292} (1992), no. 4, 685--710. 

\end{thebibliography}
\end{document}